\documentclass[12pt,a4paper,english]{smfart}
\usepackage[T2A,T1]{fontenc}
\usepackage[english]{babel}
\usepackage{ragged2e}
\usepackage{mathrsfs}
\usepackage[utf8]{inputenc}
\usepackage{amsmath}

\usepackage{smfthm, mathabx}
\usepackage{stmaryrd}
\usepackage{amssymb}
\usepackage{amsmath}

\usepackage{amsthm}
\usepackage{amsfonts}
\usepackage{graphicx}
\usepackage{enumerate}
\usepackage{changes}
\usepackage{hyperref}

\usepackage{tikz}
\usepackage{tikz-cd}

\usepackage{appendix}

\usepackage{euscript,mathrsfs}
\usepackage[utf8]{inputenc} 
\usepackage{longtable}
\usepackage{dsfont}

\usepackage[all,cmtip]{xy}
\usepackage{alltt}

\usepackage{calrsfs}

\makeatletter
\def\thm@space@setup{%
  \thm@preskip=\parskip \thm@postskip=0pt
}
\makeatother

\xyoption{all}

\usepackage{float}

\usepackage{fullpage}

\usepackage{color}

  \newcommand{\swidetilde}[1]{\smash{\widetilde{#1}}}

\newtheorem{Theorem}{Theorem}

\newtheorem{coco}{Corollary}

\author[Julian Feuerpfeil]{Julian Feuerpfeil}
\address{Dipartimento di Matematica e Applicazioni, Università di Milano-Bicocca, 20125 Milano, Italy}
\address{
 FEMTO-ST, Université Marie et Louis Pasteur, 25030 Besançon, France}
 \email{j.feuerpfeil@campus.unimib.it}

\author[Oussama Hamza]{Oussama Hamza}
\address{Institute for Advanced Studies in Mathematics, Harbin Institute of Technology,
 Harbin, China 150001}
\email{ohamza3@uwo.ca}

\author{Donghyeok Lim}
\address{Korea National University of Education, Department of Mathematics Education, 28173 Cheongju,South Korea}
\email{donghyeokklim@gmail.com}

\title{Tame Galois Groups, Linking Numbers and Mildness}

\thanks{The authors are thankful to Christian Maire, J{\'a}n Min{\'a}{\v c} and Thomas Weigel for discussions and their support. They are especially grateful to Christian Maire for his support and strong interest in this paper. The first two authors are also grateful to Claudio Quadrelli, Simone Blumer and John Labute for many insightful discussions on the topic and references. The second author also thanks Elyes Boughattas, Keping Huang, Yuanyuan Jing, Ting Lu, Bakhrom Omirov, Jun Wang and Zhizheng Yu for discussions. The last two authors are grateful to Henrique Souza for helpful comments and to Heon Lee for his support. The third author was supported by the National Research Foundation of Korea (NRF) grants No. RS-2024-00462910.}

\newcommand{\Q}{\mathbb{Q}}
\newcommand{\F}{\mathbb{F}}
\newcommand{\Z}{\mathbb{Z}}
\newcommand{\NN}{\mathbb{N}}

\newcommand{\RB}{\CyB}

\newcommand{\Supp}{\mathrm{Supp}}

\newcommand{\Gov}{\mathrm{Gov}}

\def\p{{\mathfrak p}}
\def\q{{\mathfrak q}}
\def\t{{\mathfrak t}}

\def\grad{{\rm Grad}}

\def\Gal{{\rm Gal}}

\def\bX{{\mathbf{X}}}
\def\bE{{\mathbf{E}}}

\def\GG{{\mathcal G}}

\def\O{{\mathcal O}}

\def\PP{{\mathcal P}}
\def\J{{\mathcal J}}
\def\E{{\mathcal E}}
\def\Ll{{\mathcal L}}

\def\A{{\mathcal A}}
\def\AA{{\mathbb A}}

\def\I{{\mathcal I}}

\def\L{{\rm L}}

\def\p{{\mathfrak p}}

\DeclareFontFamily{T2A}{cmr}{}
\DeclareFontShape{T2A}{cmr}{m}{n}{<-> cmr10}{}

\DeclareSymbolFont{cyrletters}{T2A}{cmr}{m}{n}
\DeclareMathSymbol{\Be}{\mathalpha}{cyrletters}{102}
\font\rurm=wncyr10 scaled \magstep1
\def\sha{{{\textnormal{\rurm{Sh}}}}}
\def\CyB{{{\textnormal{\rurm{B}}}}}

\def\Sha{{\sha}}

\begin{document}

\begin{abstract}
Let~$p$ be an odd prime and let~$S$ be a set of tame primes. We denote by~$G_S$ the Galois group of the maximal pro-$p$ extension of $\Q$ unramified outside~$S$.

We prove that for every finite set of tame primes~$S_0$ with~$|S_0|\geq 2$, there exists a set~$S_1$ consisting of two tame primes such that~$G_{S_0\cup S_1}$ has cohomological dimension~$2$. This refines a result of Labute. More generally, we establish an analogous result for number fields not containing a primitive $p$-th root of unity, under a suitable splitting condition.  

Our approach answers a question of Labute, from his seminal paper on mild groups, and combines weighted Zassenhaus filtrations, graph-theoretic methods, and Koch-type presentations. As an application, we solve several cohomological Galois inverse problems with prescribed ramification and splitting. We also provide numerical examples and statistics.
\end{abstract}




\maketitle

\section*{Introduction}
\subsection*{Context}
Let~$p$ be an odd prime and~$K$ be a number field. A prime~$\q$ of~$K$ is said to be \emph{tame} if $N(\q)\equiv 1\pmod p$. For a set of finite tame primes $S$ and $T$ a set of primes disjoint from $S$, we denote by~$G_{K,S}$ (resp.\ $G_{K,S}^T$), the  Galois group of the maximal pro-$p$ extension of $K$ unramified outside~$S$ (and resp.\ totally split in $T$, with~$T$ a set of primes disjoint from~$S$). For~$K\coloneq \Q$, we use the notation~$G_S^T$.

Following the work of Golod and Shafarevich on class field towers, the structure of the groups~$G_{K,S}^T$ has become a central topic in Galois theory. Using their famous result, several infinite (indeed nonanalytic) groups have been constructed as~$G_{K,S}^T$, with various arithmetic applications (\cite{hajir2021cutting, hajir2020, hajir2025}). However, despite the abundance of such examples, relatively little is known about their structure. By class field theory, the groups~$G_{K,S}^T$ are FAB, i.e.\ every open subgroup has finite abelianization.
The celebrated tame Fontaine--Mazur conjecture (see \cite[Conjecture 5b]{fontaine1995geometric}) also predicts that every $p$-adic analytic quotient of~$G_{K,S}^T$ is finite. This highlights the importance of understanding such groups in the context of Galois representations.

Koch was the first to determine a complete presentation for the groups~$G_S$ (see~\cite{Koch65,Koch}), which opened many directions for further investigation. Building on these results and the concept of strongly free sequences and mild groups introduced by Anick~\cite{Anick1987, AN0}, Labute~\cite{Labute} was able to construct several groups~$G_{S}$ of cohomological dimension~$2$. He showed~\cite[Corollary~$6.3$]{Labute} that every finite set of tame primes~$S$ can be extended to a tame set~$\swidetilde{S}$ of size~$2\cdot |S|$ such that~$G_{\swidetilde{S}}$ is of cohomological dimension~$2$. For the case~$p=2$, we refer to~\cite{Labute-Minac, FOR}. Later, Schmidt~\cite{schmidt2008uber,Schmidt2006,Schmidt2007} interpreted these results in terms of the étale homotopy type of~$\mathrm{Spec}(\O_{K,S})$, and showed that, for the sets~$S$ considered by Labute, $\mathrm{Spec}(\O_{K,S})$ is a~$K(\pi,1)$ space for~$p$ in the context of arbitrary number fields. We refer to~\cite{genesislabute} for an historical survey. For other applications related to mild groups in arithmetic, we refer to~\cite{split, hamza2023zassenhaus, hamza2025masseyproductsunipotentextensions, MaireFAB}. Currently, it is still unknown whether there exists a group~$G_{K,S}^T$ of finite cohomological dimension greater than~$2$.

\subsection*{Statement of the results}
We assume throughout the paper that~$K$ does not contain~$\zeta_p$, a primitive root of unity. We denote by~${\rm cl}(K)$ the class group of~$K$ and~$d_K'\coloneq r_K+d_p{\rm cl}(K)$, where $r_K$ is the torsion-free rank of $\mathcal{O}_K^\times$ and $d_p{\rm cl}(K)$ denotes the $p$-rank of the class group.
We fix~$T_K$ a minimal set of primes generating the $p$-part of the class group ${\rm cl}(K)$. Note that~$T_K$ always exists using the Chebotarev Density theorem and~$|T_K|=d_p{\rm cl}(K)$. For instance, if~$p$ is large enough or~$K=\Q$, then~$T_K=\emptyset$. Let~$S_1$ and~$S_2$ be two sets of primes. We introduce
\begin{align*}
    V_{S_1}^{S_2}(K)\coloneq \big\{a\in K^{\times }\mid a\in K_\q^{\times p}\text{ for }\q\in S_1\text{ and }p\mid \nu_\mathfrak{t}(a)\text{ for }\mathfrak{t}\not\in S_2\big\}/K^{\times p},
\end{align*} 
with~$\nu_\mathfrak{t}$ the normalized valuation associated to~$\mathfrak{t}$. This allows us to define the Koch-module~$\RB_{S_1}^{S_2}(K)$ as the dual of~$V_{S_1}^{S_2}(K)$.

We say that~$S'$ is a \emph{Koch set} if it is tame, disjoint from~$T_K$, and satisfies $(a)$~$\RB_{S'}^{T_K}(K)=1$, and $(b)$~$G_{K,S'}^{T_K}=1$.
A finite tame set~$S$ \emph{admits a Koch set} if~$S$ contains a Koch set~$S'$ and is disjoint from~$T_K$. We show in Lemma~\ref{lem:Existence of Koch Tuples} that a set~$S$ can always be extended to contain a Koch set. Furthermore, Proposition~\ref{prop:Probabilty for existence of Koch set} implies that for $S$ large enough one can expect $S$ to already admit a Koch set with high probability.  If $S$ admits a Koch set, then an analog of the presentation of~$G_S$ given by Koch exists for~$G_{K,S}^{T_K}$. This was observed by Liu in~\cite[Thm.~$1.1$]{Liu24}. We call these presentations \lq\lq extended\rq\rq\, Koch-type presentations. For a precise definition, we refer to Section~\ref{ssec:Koch presentations and linking data}. A crucial piece of data of such a presentation is a linking type function denoted~$\mu_S^{T_K}$.
Note that when~$K=\Q$, the group~$G_S$ has a Koch-type presentation and~$\mu_S$ is the usual linking map as considered by Labute \cite{Labute} and many others (see for example \cite{Schmidt2006,Salle, MaireFAB}). 
Using an argument developed by Maire and Sankara in~\cite{MaireSankara}, we positively answer a question posed by Labute, see Theorem~\ref{thm:Realization of extended linking data}, and extend it in our context. It roughly states, that every function~$\mu$ with the right domain can be realized as $\mu_S^T$. Theorem~\ref{thm:Realization of extended linking data}, coupled with well-chosen weighted Zassenhaus filtrations, is fundamental in our work. We state our main result.

\begin{Theorem}\label{first theorem}
    For any finite set of tame primes~$S_0$ such that~$|S_0|\geq 2$, there exists a finite set of tame primes~$S_1$ such that~$|S_1|= 2+2d_K'$ and~$G_{K,S_0\cup S_1}^{T_K}$ has cohomological dimension~$2$. Furthermore, the deficiency of~$G_{K,S_0\cup S_1}^{T_K}$ is~$d_K'$.
\end{Theorem}

A notable feature of Theorem~\ref{first theorem} is that the size of~$S_1$ depends only on $K$, and is completely independent of the size of~$S_0$. The careful choices on weighted Zassenhaus filtrations used simultaneously with Theorem~\ref{thm:Realization of extended linking data}, and combined with arguments from Right Angled Algebras developed by the second author in~\cite{hamza2023extensions} (see also Remark~\ref{rem:sharper result first theorem}), are essential in our proof. Previously, Forré~\cite{FOR} used judicious choices of filtrations to infer results on mildness for the case~$p=2$. Gärtner~\cite[Remark~$2.12, (ii)$]{GARTNER2015788} also noted that mildness of a presentation may depend on the choice of the filtration. 

As a direct consequence, for the case~$K=\Q$, we refine~\cite[Corollary~$6.3$]{Labute}.
\begin{coco}\label{coco Labute}
    Let~$K=\Q$, and~$S_0$ be a finite set of tame primes such that~$|S_0| \geq 2$. Then, there exists a set~$S_1$ of two tame primes such that~$G_{S_0\cup S_1}$ has cohomological dimension~$2$.
\end{coco}

A similar approach allows us to study some cohomological inverse problems in the context of restricted ramification with splitting. We refer to Section~\ref{sec:Genus Theory}. Namely, we realize some quadratic algebras as cohomology algebras of groups~$G_{K,S}^T$ for some well chosen~$S$ and~$T$. For this purpose, we use graph theory. It has been extensively used to study which pro-$p$ groups arise as maximal pro-$p$ Galois groups of fields. Snopce and Zalesskii~\cite{snopce2020right}, together with Cassella and Quadrelli~\cite{cassella2021right}, characterized which pro-$p$ RAAGs are maximal pro-$p$ Galois groups. They resolved several conjectures in this class.
For further work, see~\cite{blumer, hamza2025maximal2extensionspythagoreanfields, hamza2023extensions, leoni2024zassenhaus}.

\subsection*{Some numerical examples}
The results and proofs of Theorem~\ref{first theorem}, and Section~\ref{sec:Genus Theory} can be made computationally effective to produce concrete examples. We used OSCAR (see \cite{OSCAR}) for our computations. The source code and a Jupyter Notebook, containing all the examples, can be found on the first author's GitHub under \cite{Feuerpfeil2026Code}.

For example, by virtue of Theorem~\ref{first theorem} we show in Subsection~\ref{ssec:Examples mild groups} that the group~$G_{S}$, for~$p=3$ and the set~$S\coloneq \{7,13,19,37,10639,826093\}$, is mild with respect to a weighted Zassenhaus filtration and therefore of cohomological dimension $2$. 
We give further examples of this behavior also in the case of more general number fields in Section~\ref{ssec:Examples mild groups}.

Furthermore, Subsection~\ref{ssec:Examples for 1 RAAGs} produces examples of groups $G_S$ whose~$\F_p$-cohomology ring~$H^\bullet(G_S)$ comes from
graph theory. Let~$\A$ be the graded algebra over~$\F_3$ presented by four generators~$\{X_1,X_2,X_3,X_4\}$ and~$12$ relations: 
    $$\{X_i^2, \quad X_uX_v+X_vX_u,\quad X_1X_3, \quad X_2X_4, \quad 1\leq i \leq 4, \quad 1\leq u< v\leq 4\}.$$
Then for~$p=3$ and~$S\coloneq \{7,13,181,5563\}$ we have $H^\bullet(G_{S})\simeq \A$. Subsection~\ref{ssec:nonRAAG} gives more examples for other number fields and for certain quadratic algebras not coming from graph theory.

\subsection*{Organization of the paper}
The first section introduces all of the group-theoretical arguments required for our results. In particular, Subsection~\ref{sec:cohomology and graphs} introduces the notion of~$(\bX,e)$-RAAGs, which are not in general RAAGs, but nevertheless have the same $(\bX,e)$-filtration as RAAGs associated to certain graphs. The second section introduces all arithmetic arguments needed in this paper. In particular, it answers a question of Labute by proving Theorem~\ref{thm:Realization of extended linking data}, using an argument from~\cite{MaireSankara}. The third section is devoted to the proof of Theorem~\ref{first theorem}. The main group theoretical argument relies on suitable choices of~$(\bX,e)$-filtrations, which yield $(\bX,e)$-RAAGs whose underlying graphs are triangle-free. The fourth section applies graph theory to the cohomological inverse problem. The last section focuses on numerical examples.

\section{Group theoretic and algebraic preliminaries}
\label{sec:Groups And Algebras}

In this section we introduce general group-theoretic tools and techniques that we need to prove most of our results. For general references, we refer to~\cite{LAZ, split, hamza2023extensions, hamza2023zassenhaus, hamza2026filtrationscohomologygraphproducts}. 

\subsection{Filtered and graded (Lie)-algebras}

Set an integer~$d$ and a finite set~$\bX\coloneq \{x_1,\dots, x_d\}$. We denote by~$\F_p\langle\!\langle X_1,..,X_d\rangle\!\rangle $ the algebra of noncommutative series in the variables~$X_1,\dots, X_d$ over~$\F_p$.
A basis of the topological~$\F_p$-vector space~$\F_p \langle\!\langle X_1,\dots, X_d\rangle\!\rangle$ is given by monomials~$X^{\alpha}\coloneq \smash{X_{i_{\alpha_1}}^{\alpha_1}\dots X_{i_{\alpha_n}}^{\alpha_n} }$, where~$\alpha$ is an~$n$-tuple and~$i_{\alpha_k}$ is an element in~$\{1,\dots,d\}$.
Consequently, every element~$x$ in~$\F_p \langle\!\langle X_1,\dots, X_d\rangle\!\rangle$ can be written as an infinite sum~$x\coloneq \sum_{\alpha}{a_\alpha X^\alpha}$, with~$a_{\alpha}\in \F_p$.

For a given~$e\coloneq (e_1,\dots, e_d)\in \NN^d$ we denote by~$E_e(\bX)$ the set~$\F_p\langle\!\langle X_1,..,X_d\rangle\!\rangle$ endowed with the filtration~$\{E_{e,n}(\bX)\}_{n\in \NN}$ induced by the following degree function:
\begin{align*}
\omega_e \Big({\sum}_{\alpha} a_\alpha \cdot X^\alpha \Big) \coloneq \min_{\alpha} \{ \alpha_1 e_{i_{\alpha_1}}+\dots +\alpha_n e_{i_{\alpha_n}} \}.
\end{align*}
We define~$\E_e(\bX)\coloneq \bigoplus_{n\in \NN} \E_{e,n}(\bX)$, with $\E_{e,n}(\bX)=E_{e,n}(\bX)/E_{e,n+1}(\bX)$. Note that~$\E_e(\bX)$ is isomorphic to the (free) graded algebra~$\F_p \langle X_1,\dots ,X_d\rangle$, where every~$X_i$ is of degree $e_i$.
If~$x$ is a nonzero element in~$E_e(\bX)$, we define its initial form~$\overline{x}$ as the image of~$x$ in~$$\E_{e,\omega_e(x)}(\bX)=E_{e,\omega_e(x)}(\bX)/E_{e, \omega_e(x)+1}(\bX).$$
We fix an ordering on the set $\{X_i\}_{i=1}^d$. We say that~$X^{\alpha}<X^{\beta} $, if $\omega_e(X^{\alpha})>\omega_e(X^\beta)$ and if we have equality, we use the lexicographic order. 
Define~$\widehat{x}\coloneq \max_{a_\alpha \neq 0}\{X^{\alpha}\}$. 

Let $I$ be a closed two-sided ideal of $E_e(\bX)$. We endow~$I$ with the filtration~$\{I_n\coloneq I\cap E_{e,n}(\bX)\}_{n\in \NN}$. The algebra~$A\coloneq E_e(\bX)/I$ is endowed with a filtration that we call the quotient filtration (of~$E_e(\bX)$ by~$I$, see~\cite[Chapitre~I,~$2.1.7$]{LAZ}). We denote it by~$\{A_n\}_{n\in \NN}$. We define~$\grad(A)\coloneq \bigoplus_{n\in \NN} \grad_n(A)$, where~$\grad_n(A)\coloneq A_n/A_{n+1}$. We also introduce the Hilbert series of~$A$ (resp.\ of a graded algebra~$\A\coloneq \E_e(\bX)/\mathcal{J}=\bigoplus_{n\in \NN}\A_n$, for some graded ideal~$\J$) by: 
 $$A(t)\coloneq \sum_{n\in \NN}\dim_{\F_p}(\grad_n(A))t^n, \quad \text{and} \quad \A(t)\coloneq \sum_{n\in \NN}\dim_{\F_p}(\A_n)t^n.$$

\subsection{Magnus's isomorphism and applications in group theory}

We assume that~$G$ is a finitely presented pro-$p$ group presented by~$d$ generators~$\{\tau_x|\quad x\in \bX\coloneq \{x_1,\dots, x_d\}\}$ and~$r$ relations~$\{l_1,\dots, l_r\}$. We denote this presentation~$\PP$. Let~$F(\bX)$ be the free pro-$p$ group on~$\{\tau_x, x\in \bX\}$ and~$R$ be its closed normal subgroup on~$\{l_1,\dots, l_r\}$. We have an exact sequence:
$$1\to R \to F(\bX) \to G \to 1.$$

Magnus (see for instance~\cite[Appendice~$A.3$]{LAZ}) constructed an injective group homomorphism~$\varPsi_{F(\bX)}:F(\bX)\to \F_p\langle\!\langle X_1,...,X_d\rangle\!\rangle^{\times}$ induced by $\tau_{x_i}\mapsto 1+X_i$. This induces a filtration on~$F(\bX)$, which we call~$(\bX, e)$-filtration, defined by:$$F_{e,n}(\bX)\coloneq \{f\in F(\bX), \quad \varPsi_{F(\bX)}(f)-1\in E_{e,n}(\bX)\}.$$
We also define~$\Ll_e(\bX)$ the free graded $p$-restricted Lie-algebra on~$\{X_1,\dots, X_d\}$ over $\F_p$, where every~$X_i$ is endowed with weight~$e_i$.
Magnus isomorphism~\cite[Chapitre~II, Théorème~$3.2.5$]{LAZ} also allows us to construct an isomorphism~$$\grad(\varPsi_{F(\bX)})\colon \bigoplus_{n\in \NN} F_{e,n}(\bX)/F_{e,n+1}(\bX) \to \Ll_e(\bX).$$

We write~$I_e(R)$ for the closed two-sided ideal of~$E_e(\bX)$ generated by~$\{ \varPsi_{F(\bX)}(l)-1\}_{l\in R}$. We set ~$E_e(\bX,G)\coloneq E_e(\bX)/I_e(R)$, equipped with the quotient filtration~$\{E_{e,n}(\bX,G)\}_{n\in \NN}$. The map $\varPsi_{F(\bX)}$ induces $\varPsi_{\bX,G}\colon G \to E_e(\bX,G)^{\times}$ mapping $\tau_{x_i}$ to the class of~$1+X_i$. We denote by~$\I_e(R)$ the ideal of~$\E_e(\bX)$ isomorphic to~$\bigoplus_{n\in \NN}I_{e,n}(R)/I_{e,n+1}(R)$. This allows us to define:
$$E_e(\bX,G)\coloneq E_e(\bX)/I_e(R), \quad \text{and} \quad \E_e(\bX,G)\coloneq \E_e(\bX)/\I_e(R) \coloneq \bigoplus_{n\in \NN} \E_{e,n}(\bX, G).$$
 We define, for every integer~$n$, the~$(\bX, e)$-filtration of~$G$ by
$$G_{e,n}(\bX)\coloneq \{g\in G, \varPsi_{\bX,G}(g)-1\in E_{e,n}(\bX,G)\}.$$
If~$e=(1,\dots, 1)$, then~$\{G_{e,n}(\bX)\}_{n\in \NN}$ is independent of~$\bX$ (see~\cite[Chapitre~II, section~$3.2$]{LAZ}). In that case, we omit ~$e$ and~$\bX$, and write~$E_n$, $G_n$, and ~$F_n$. This filtration coincides with the Zassenhaus filtration (see~\cite[Theorem~$12.9$]{DDMS}) defined inductively~by
$$G_1\coloneq G, \quad \text{and} \quad G_{n}\coloneq G_{\lceil \frac{n}{p} \rceil}^p \prod_{i+j=n} \lbrack G_i, G_j\rbrack.$$

We endow~$R$ with the induced filtration given by $R_{e,n}\coloneq F_{e,n}(\bX)\cap R$. Let us observe that the filtration~$\{G_{e,n}(\bX)\}_{n\in \NN}$ coincides with the quotient filtration 
$\{F_{e,n}(\bX)\}_n$ by~$\{R_{e,n}\}_n$. We define~$\J_e(R)$ the Lie ideal of~$\Ll_e(\bX)$ given by~$\grad(\varPsi_{F(\bX)})\left(\bigoplus_{n\in \NN} R_{e,n}/R_{e,n+1}\right)$, and we define:
$$\Ll_e(\bX, G)\coloneq \Ll_e(\bX)/\J_e(R)\coloneq \bigoplus_{n\in \NN}\Ll_{e,n}(\bX, G).$$
Note that~$\Ll_{e,n}(\bX, G)\simeq G_{e,n}(\bX)/G_{e,n+1}(\bX)$.


For~$1\leq i \leq d$, we define~$n_{e,i}\coloneq \omega_e(\psi_{F(\bX)}(l_i)-1)$, and~$\rho_{e,i}$ the image of~$\psi_{F(\bX)}(l_i)-1$ in~$\E_{e,n_{e,i}}(\bX)$. We denote by~$\I_e(\PP)$ the two sided ideal generated by~$\{\rho_{e,i}\}_{i=1}^r$, and:
$$\E_e(\bX,\PP)\coloneq \E_e(\bX)/\I_e(\PP), \quad \text{and} \quad \E_e(\bX, \PP,t)\coloneq \E_e(\bX, \PP)(t).$$
Note that~$\rho_{e,i}$ can also be seen as an element of~$\Ll_e(\bX)$ through the injection (given by universal enveloping algebra theory) $\Ll_e(\bX)\hookrightarrow \E_e(\bX)$. Thus, we define~$\J_e(\PP)$ the restricted Lie-two sided ideal of~$\Ll_e(\bX)$ generated by~$\rho_{e,i}$, and~$\Ll_e(\bX,\PP)\coloneq \Ll_e(\bX)/\J_e(\PP)$.
\begin{lemm}\label{beg rel}
    We have two epimorphisms:
    $$\psi_{\PP}\colon \E_e(\bX, \PP)\twoheadrightarrow \E_e(\bX,G), \quad \text{and} \quad \varphi_\PP\colon \Ll_e(\bX, \PP)\twoheadrightarrow \Ll_e(\bX,G); \quad \text{such that } X_i\mapsto X_i.$$
\end{lemm}

\begin{proof}
    We observe, for every~$1\leq i \leq d$, that~$\rho_{e,i}$ is an element in~$\I_e(R)$ and~$\J_e(R)$. Thus we have an inclusion of~$\I_e(\PP)$ (resp.\ $\J_e(\PP)$) in~$\I_e(R)$ (resp.\ $\J_e(\PP)$), and so the desired surjection.
\end{proof}



\subsection{Mild presentations}
We continue to use the notation introduced in the previous section. In particular, $G$ denotes a finitely presented pro-$p$ group and $\PP$ a presentation of $G$.
Little is known about the structure of $G$ when~$\psi_\PP$ and~$\varphi_\PP$ fail to be isomorphisms. We therefore focus on situations in which they are isomorphisms. For this purpose, we use the following definition of mild groups (cf.~\cite[Lemma~1.3]{FOR}). 


\begin{defi}
    We say that the group~$G$ has a mild presentation~$\PP$ with respect to~$(\bX,e)$, if:
    $$\E_e(\bX, \PP,t)=\frac{1}{1-\sum_{i=1}^d t^{e_i}+\sum_{j=1}^r t^{n_{e,j}}}.$$
\end{defi}


\begin{theo}
Assume that~$G$ has a mild presentation~$\PP$ with respect to some~$(\bX,e)$ filtration. Then~$\psi_{\PP}$ and~$\varphi_\PP$ are isomorphisms. Furthermore, the group~$G$ has cohomological dimension~$2$.
\end{theo}

\begin{proof}
See Labute~\cite[Theorem~$1.2$]{Labute}.
\end{proof}

Let us recall that the notion of combinatorially free families provides an effective criterion to check mildness.

 \begin{defi}
We consider a family $\{X_{u_i}X_{v_i}\}_{i=1}^r $ of quadratic monomials. This family is said to be combinatorially free if for every~$1\leq i,j \leq r$ we have~$X_{v_i}\neq X_{u_j}$.
\end{defi}


\begin{coro}\label{hilbseriescomb}
Let~$G$ be a pro-$p$ group on~$\{\tau_x, x\in\bX\}$ generators and presented by~$\PP$. Consider an~$(\bX,e)$-filtration on~$F(\bX)$. If the family $\widehat{\rho_e}\coloneq \{\widehat{\rho_{e,i}}\}_{i=1}^r$ is quadratic and combinatorially free, then the presentation~$\PP$ is mild for the~$(\bX,e)$-filtration.
\end{coro}

\begin{proof}
See~\cite[Theorem~$2.6$]{FOR}.
\end{proof}

The next subsection provides concrete examples.

\subsection{Right Angled Artin Algebras and group applications}\label{sec:cohomology and graphs}

Let~$\Gamma\coloneq (\bX, \bE)$ be an (undirected) graph, and~$e$ some weights. For now, we fix a group~$G$, presented by~$\PP$ with set of generators~$\bX$ and relations~$\{l_{x_u,x_v}\}_{\{x_u,x_v\}\in \bE}$. 

\begin{defi}[$(\bX,e)$-RAAG]
    We say that a group~$G$ is~$(\bX,e)$-RAAG for the graph~$\Gamma$, if the relations~$\{l_{x_u,x_v}\}_{\{x_u,x_v\}\in \bE}$ satisfy:
$$l_{x_u,x_v}\equiv \lbrack \tau_{x_u},\tau_{x_v}]^{\mu(x_u,x_v)}\pmod{F_{e,e_u+e_v+1}(\bX)}$$ for some $\mu(x_u,x_v)\in \F_p^\times$. When we consider the Zassenhaus filtration, we directly say~$1$-RAAG.
\end{defi}
\begin{rema}
    By replacing $l_{x_u,x_v}$ by a suitable power, one can assume that $\mu(x_u,x_v)=1$, for every~$\{x_u,x_v\}\in \bE$.
\end{rema}


If~$G$ is~$(\bX,e)$-RAAG for~$\Gamma$, then the initial form of the relations can now be determined as~$\rho_{e, x_u,x_v}=\mu(x_u,x_v)\lbrack X_u, X_v\rbrack=\mu(x_u,x_v)(X_uX_v-X_vX_u).$
We define
$$ \E_e(\bX, \Gamma) \coloneq \E_e(\bX)/\I_e(\Gamma), \quad \Ll_e(\bX, \Gamma) \coloneq \Ll_e(\bX)/\J_e(\Gamma)\quad \text{and} \quad \E_e(\bX, \Gamma,t) \coloneq \E_e(\bX, \Gamma)(t),$$
where~$\I_e(\Gamma)$ (resp.\ $\J_e(\Gamma)$) is the ideal of~$\E_e(\bX)$ (resp.\ $\Ll_e(\bX)$) generated by~$\lbrack X_u, X_v\rbrack$ for~$\{x_u,x_v\}\in \bE$. 
If $\PP$ is the presentation of $G$ associated with $\bX$ and $\{l_{x_u,x_v}\}_{\{x_u,x_v\}\in \bE}$, then it is easy to observe that $\E_e(\bX,\Gamma)\simeq \E_e(\bX, \PP)$ and $\Ll_e(\bX, \Gamma)\simeq \Ll_e(\bX, \PP)$. 

We say that a graph~$\Gamma\coloneq (\bX,\bE)$ is \emph{bipartite} if there is a nontrivial partition~$\bX\coloneq U \amalg V$, such that there are no edges between two vertices of~$U$, and no edges between two vertices of~$V$.
\begin{theo}\label{mildGammaRAAG}
Assume that~$\Gamma\coloneq (\bX, \bE)$ is bipartite, $e$ and $G$ as above. Then the presentation~$\PP$ is mild for the~$(\bX, e)$-filtration. So~$G$ has cohomological dimension~$2$ and we have isomorphisms:
$$\varphi_{\PP}\colon \Ll_e(\bX, \Gamma) \simeq \Ll_e(\bX, G), \quad \text{and} \quad \psi_\PP\colon \E_e(\bX, \Gamma)\simeq \E_e(\bX, G).$$ 
\end{theo}
\begin{proof}
We use~\cite[Remark~$1.2$]{hamza2023extensions} and Corollary~\ref{hilbseriescomb}.
\end{proof}


\begin{exem}\label{exem: def RAAG}
    Let~$\Gamma$ be a graph, with associated Right-Angled-Artin pro-$p$ group (short RAAG) $G_\Gamma$. This is the pro-$p$ group presented by generators~$\{\tau_{x_1},\dots, \tau_{x_d}|\quad x_i\in \bX\}$ and relations~$\lbrack \tau_{x_u}, \tau_{x_v}\rbrack$, for~$\{x_u,x_v\}\in \bE$. As noted in~\cite[Proposition~$1.7$]{hamza2023extensions}, for any~$(\bX,e)$-filtration, the group~$G_\Gamma$ is~$(\bX,e)$-RAAG. 

     If $G$ is a $1$-RAAG for some graph~$\Gamma$, then we have $G_\Gamma/(G_\Gamma)_3\cong G/G_3$, thus they agree on a finite level. We refer to Section~\ref{sec:Genus Theory} for further details.

     The group~$G_\Gamma$ has several nice properties and is well studied. We refer for instance to~\cite{bartholdi2020right, lorensen2010groups, hamza2023extensions, hamza2026filtrationscohomologygraphproducts}. 
\end{exem}
Before stating the next result, we introduce some definitions and notations. We define the algebra~$\A(\Gamma)$ as the quadratic algebra over~$\F_p$ presented by generators~$ \{X_1,\dots, X_d\}$ and relations:
  \begin{itemize}
\item~$X_iX_j$ when~$\{x_i,x_j\}\notin \bE$,
\item~$X_uX_v+X_vX_u$ for~$x_u,x_v$ in~$\bX$.
\end{itemize} 
Observe that~$\dim_{\F_p}\A_n(\Gamma)=c_n(\Gamma)$, where~$c_n(\Gamma)$ is the number of~$n$-cliques of~$\Gamma$, i.e.\ complete subgraphs of~$\Gamma$ with~$n$ vertices. We also denote by~$H^\bullet(G)$ the (continuous) cohomology graded algebra of~$G$ with coefficients in~$\F_p$.
\begin{prop}\label{1-RAAG}
Assume that~$\Gamma$ is triangle-free, and~$G$ is a $1$-RAAG for~$\Gamma$. Then the presentation~$\PP$ is mild for the~$(\bX, e)$-filtration with~$e=(1,\dots ,1)$. So we have $\Ll(G)\simeq \Ll(\Gamma)$, and the following isomorphism:
$$H^\bullet(G)\simeq \A(\Gamma), \quad \text{and} \quad h^n(G)\coloneq\dim_{\F_p}H^n(G)=c_n(\Gamma).$$
\end{prop}

\begin{proof}
We denote by~$d$ the number of vertices of~$\Gamma$ and by~$r$ the number of edges. We have, from~\cite[Theorem~$1.5$]{bartholdi2020right}:
    $${\E(\PP,t)=}\E(\Gamma ,t)=\frac{1}{1-dt+rt^2}.$$
Thus, the presentation~$\PP$ is mild for the~$(\bX, e)$-filtration with~$e=(1,\dots ,1)$.
    Since~$\E(\Gamma)$ is a Koszul algebra, we conclude from~\cite[Proposition~$1$]{hamza2023extensions}. 
\end{proof}

\begin{rema}
Precisely, the last proposition gives us:
$$H^0(G)\simeq \F_p, \quad H^1(G)\simeq \bigoplus_{x_u\in \bX} X_u \F_p \quad \text{and} \quad H^2(G)\simeq \bigoplus_{u<v, \{x_u,x_v\}\in \bE}X_uX_v \F_p.$$
Furthermore, we have~$X_u X_v=-X_vX_u$ for~$x_u,x_v\in \bX$ in~$H^2(G)$. Finally~$X_uX_v=0$ in~$H^2(G)$ precisely when~$\{x_u,x_v\}$ is not in~$\bE$.
\end{rema}

\section{Koch-type presentations and linking numbers}
\label{sec:Koch presentation and linking numbers}
We fix an odd prime $p$. Throughout the paper, we assume that $K$ is a number field, not containing $\zeta_p$. We denote by $r_K$ the $\Z$-rank of the unit group of $\mathcal{O}_K$. If $(r_1,r_2)$ is the signature of $K$, then $r_K=r_1+r_2-1$ by the Dirichlet unit theorem. We also define~${\rm cl}(K)$ to be the class group of~$K$ and~$d_p{\rm cl}(K)$ to be its~$p$-rank. This allows us to introduce~$d_K'\coloneq r_K+d_p{\rm cl}(K)$. We fix a set of primes~$T_K$ of~$K$ with Frobenii generating the~$p$-part of the group~${\rm cl}(K)$. Note that~$|T_K|=d_p{\rm cl}(K)$.

Let $S_1$, $S_2$ be sets of primes of~$K$. Then we denote 
\begin{align*}
    V_{S_1}^{S_2}(K)\coloneq \big\{a\in K^{\times }\mid a\in K_\q^{\times p}\text{ for }\q\in S_1\text{ and }p\mid \nu_\mathfrak{t}(a)\text{ for }\mathfrak{t}\not\in S_2\big\}/K^{\times p},
\end{align*} 
with~$\nu_\mathfrak{t}$ the normalized valuation associated to~$\mathfrak{t}$. The dual of~$V_{S_1}^{S_2}(K)$ is denoted by~$\RB_{S_1}^{S_2}(K)$ (see~\cite[Def. 10.7.8]{NSW}). The set $\RB_{S_1}^{S_2}$ is important because it governs the Tate--Shafarevich group in Galois cohomology and is often computationally accessible (cf. proof of \ref{theo-Liu}, \cite{HMRSha}). We note that if $S_1\subseteq S_1'$ and $S_2\subseteq S_2'$, then we have the following canonical surjections:
\begin{align*}
    \RB_{S_1}^{S_2}\twoheadrightarrow \RB_{S_1'}^{S_2},\qquad  \RB_{S_1}^{S_2'}\twoheadrightarrow \RB_{S_1}^{S_2},\quad \text{and}\quad \RB_{S_1}^{S_2'}\twoheadrightarrow \RB_{S_1'}^{S_2}.
\end{align*}

By \cite[Thm. 10.7.10]{NSW}, if~$S$ and~$T$ are disjoint sets of primes, then the dimension of~$\RB_S^{T}(K)$ determines the minimal number~$d(G_S^T)$ of generators of~$G_S^T$. More precisely, we have the following:
\begin{prop}
\label{prop:rank of G_S^T}
    Let $S$ be a set of tame primes of $K$ and $T$ a disjoint set of primes of $K$. Then $d(G_S^T)=|S|+\dim_{\F_p} \RB_{S}^{T}(K)-r_K-|T|$.
\end{prop}
\subsection{Koch sets}
\label{ssec:Koch tuples}
In general, knowledge of $\RB_S^{T}$ is not sufficient to determine the minimal number of relations of $G_S^T$, let alone the concrete relations. In favorable situations, as considered by Liu in~\cite{Liu24} or Salle in~\cite{Salle} one can determine the relations explicitly. We summarize their considerations in the form of \emph{Koch sets}. More precisely, fix a minimal set of generators $T_K$ of the $p$-part of $\mathrm{cl}(K)$. We say that a finite set of tame primes~$S'$ of~$K$ is \emph{Koch} if~$S'\cap T_K=\emptyset$ and
\begin{align*}
    \RB_{S'}^{T_K}(K)=1,\qquad \text{and}\qquad G_{K,S'}^{T_K}=1.
\end{align*}
We say that a finite set~$S$ of tame primes of~$K$ \emph{admits a Koch set} if~$S\cap T_K=\emptyset$ and there is $S'\subseteq S$ such that $S'$ is a Koch set. 

For an integer $n$, we define $K_n\coloneq K(\zeta_{p^n})$. Before showing the existence of Koch sets, we fix some notations which we will use frequently. Given two disjoint sets of primes $S$ and $T$ we set
\begin{align*}
    \Gov^T_S(K)\coloneq K_1\Big(\sqrt[p]{V_S^{T}}\Big).
\end{align*}
By Kummer duality it is not hard to see that $\Gal(\Gov_S^T(K)/K_1)\cong \RB_{S}^{T}$. If $T=\emptyset$ (resp. $S=\emptyset$) we simply write $\Gov_S(K)$ (resp. $\Gov^T(K)$). If $K$ is clear from the context, we simply write $\Gov_S^T$.
\begin{lemm}[{\cite[Lemma 1.3]{MaireSankara}}]
\label{lem:Governing field extension generated by Frobenius}
    Fix three sets of disjoint primes $S$, $S'$, and $T$ of~$K$. Then $\Gal(\Gov_S^T/\Gov_{S\cup S'}^T)$ is generated by the Frobenii $(v_1,\Gov_{S}^T/K_1)$, where $v_1$ is an arbitrary prime of $K_1$ above $v\in S'$.
\end{lemm}
\begin{lemm}[{\cite[Prop. 2.3]{MaireSankara}}]
\label{lem:Unramified exact sequence}
    Let $S$ and $T$ be disjoint sets of primes of $K$ then there is an exact sequence
    \begin{align*}
        0\to H^1(G_S^T)\to H^1(G_S)\overset{\psi_S^T}\longrightarrow\bigoplus_{\mathfrak{t}\in T}H^1(\GG_{\mathfrak{t}}^{\rm ur})\overset{\varphi_S^T}\longrightarrow \RB_{S}^{T}\to \RB_S\to 0
    \end{align*}
    where $\GG_\mathfrak{t}^{\rm ur}\cong \Z_p$ is the Galois group of the maximal unramified pro-$p$ extension of $K_\mathfrak{t}$.
\end{lemm}

\begin{lemm}
\label{lem:Existence of Koch Tuples}
    For a given number field~$K$, let~$T_K$ be a set of primes, such that the classes form a minimal generating set of ${\rm cl}(K)/p$. Given a finite set of primes $S$, there exists a Koch set~$S'$ disjoint from $S$ of size $d_K'$.
\end{lemm}
\begin{proof}
    We first notice that $\dim_{\F_p}(\RB_\emptyset)=d_K'$. Now choose a basis $\mathcal{S}$ of $\RB_{\emptyset}$ and a set $S'$ disjoint from $S$, using the Chebotarev density theorem, such that each $\sigma_\q\coloneq (\q,\Gov_\emptyset/K)$ is contained in $\RB_{\emptyset}$ and $\{\sigma_\q\mid \q\in S'\}=\mathcal{S}$. By Lemma~\ref{lem:Governing field extension generated by Frobenius} we conclude that $\RB_{S'}=1$ and thus $H^1(G_{S'})\cong H^1(G_\emptyset)$ by \cite[Lem. 10.7.4 (i)]{NSW}, which implies that the maximal elementary abelian extension of $K$ unramified outside $S'$ is unramified.

    Since $K_{S'}^{\rm el,ab}=K_{\smash{\emptyset}}^{\rm el,ab}$, we have that the elements $(\t,K_{S'}^{\rm el,ab}/K)$ for $\t\in T_K$ generate $\Gal(K_{S'}^{\rm el,ab}/K)\cong {\rm cl}(K)/p$ and thus by Lemma~\ref{lem:Unramified exact sequence} we conclude that $\RB_{S'}^{T_K}(K)=1$. From $H^1(G_{S'}^{T_K})=0$, we deduce that $G_{S'}^{T_K}=1$ and thus $S'$ is a Koch set.
\end{proof}

\begin{rema}\label{rem:Application existence Koch tuple}
\begin{enumerate}
\item If $S'$ is a Koch set, then Proposition \ref{prop:rank of G_S^T} together with $G_{K,S'}^T=1$ implies that $|S'|=d'_K$.
\item Let~$T_K$ and~$T_K'$ be two minimal generating sets of the $p$-part of~${\rm cl}(K)$. Assume that~$S'$ is a Koch set for~$T_K$ and that~$S' \cap T'_K=\emptyset$. Then~$S'$ is also a Koch set for~$T_K'$.
\end{enumerate}
\end{rema}
We briefly discuss statistics on Koch sets. Let $T_K$ be a fixed set of primes of $K$ such that their classes form a minimal generating set of ${\rm cl}(K)/p$. Under these assumptions, we obtain the following proposition.


\begin{prop}
\label{prop:Probabilty for existence of Koch set}
Let~$d \geq 0$ be a fixed integer. For each ~$X > 0$, let~$A(X)$ consist of sets made up of~$d+d_K'$ tame primes~$\p$ of~$K$ satisfying~$\mathbf{N}\p \leq X$. Let~$B(X)$ be the subset of $A(X)$ consisting of those sets $S$ that admit a Koch set. Then we have
\begin{equation*}
    \lim_{X \to \infty} \frac{|B(X)|}{|A(X)|} = \prod_{i=0}^{d'_k-1} (1-p^{i-(d+d_K')}).
\end{equation*}

\end{prop}
\begin{proof}
For a prime ideal $\q$ of $K$, fix a prime $\q_1$ of $K_1$ above $\q$. For a finite set $S'$ of tame primes and finite set $T$ of primes with $T \cap S'=\emptyset$, $\RB_{S'}^T=1$ and $G_{K,S'}^T=1$ if and only if the set $W \coloneq\{(\q_1, \Gov^T/K_1)\}_{\q \in S'}$ is a basis of the $\F_p$-vector space $\Gal(\Gov^T/K_1)$; By Lemma \ref{lem:Governing field extension generated by Frobenius}, $\RB_{S'}^T=1$ is equivalent to $W$ spanning $\Gal(\Gov^T/K_1)$. Moreover, Proposition \ref{prop:rank of G_S^T} shows that, under this condition, $G_{K,S'}^T=1$ holds if and only if $W$ is a basis. Hence, a set $S$ admits a Koch set if and only if $\{(\q_1, \Gov^T/K_1)\}_{\q \in S}$ spans $\Gal(\Gov^T/K_1)$.

Now let $C(X)$ be the set of $(d+d_K')$-tuples $(\q_i)_i$ of tame primes~of~$K$ satisfying~$\mathbf{N}\q_i \leq X$ for all $i$, and let $D(X) \subset C(X)$ be the subset for which $\{(\q_i, \Gov^T/K_1)\}_{1 \leq i \leq d+d'_K}$ spans $\Gal(\Gov^T/K_1)$. Then we have
\begin{equation*}
 \lim_{X \to \infty} \frac{|B(X)|}{|A(X)|} = \lim_{X \to \infty} \frac{|D(X)|}{|C(X)|},  
\end{equation*}
since both $|C(X)|-|A(X)|$ and $|D(X)|-|B(X)|$ are of order $O(\pi(X)^{d+d'_K-1})$. On the other hand, the Chebotarev density theorem implies
\begin{equation*}
\lim_{X \to \infty} \, \frac{|C(X)|}{\pi(X)^{d+d'_K}} = \frac{1}{[K_1:K]^{d+d'_K}} \quad \text{and} \quad \lim_{X \to \infty} \, \frac{|D(X)|}{\pi(X)^{d+d'_K}} = \frac{N(d+d'_K, d'_K, p)}{([K_1:K] \cdot p^{d'_K})^{d+d'_K}}
\end{equation*}
where $N(a,b,p)$ denotes the number of $a \times b$ matrices over $\F_p$ of full rank. The claim now follows from the explicit formula for $N(a,b,p)$ (cf.~\cite[Theorem 7.1.5]{matrixnumber}).
\end{proof}

\begin{rema}
    If we let $d$ tend to $\infty$ in Proposition~\ref{prop:Probabilty for existence of Koch set} then the right hand side tends to $1$. Thus, one can expect in practice that any $S$, which is large enough admits a Koch set (for some~$T_K$, see Remark~\ref{rem:Application existence Koch tuple}).
\end{rema}

\subsection{Koch-type presentations and linking data}
\label{ssec:Koch presentations and linking data}
Fix a prime $p$, and two integers~$d$ and~$d'$. Given two finite sets~$\bX=\{x_1,...,x_d\}$ and~$\bX'=\{x_1',...,x_{d'}'\}$. A \emph{linking tuple} $\lambda$ is a quintuple~$\lambda\coloneq (\bX',\bX,\varepsilon,\mu,\omega)$ consisting of the following data:
\begin{itemize}
    \item a function $\varepsilon:\bX'\times \bX\to \{0,1\}$, to which we associate for each $x'\in \bX'$ a set $\Supp_\varepsilon(x')\coloneq\{x\in \bX\mid \varepsilon(x',x)=1\}$;
    \item a function $\mu:D_\varepsilon\to \F_p$ where 
    \begin{align*}
        D_\varepsilon=\bigg\{(x,y)\in (\bX\amalg \bX')\times \bX\left|{x\in\bX\text{ and }y\neq x\text{ or }\atop x\in \bX' \text{ and }y\in\bX\setminus  \Supp_\varepsilon(x)}\right.\bigg\};
    \end{align*}
    \item a sequence $\omega\coloneq (\omega_z)_{z\in \bX'\amalg \bX}$ with $0\neq \omega_z \in p\Z_p$ for $z\in \bX\amalg \bX'$.
\end{itemize}
If $\bX'=\emptyset$, we simply write $(\bX,\mu,\omega)$ instead of the full quintuple. 
\begin{defi}

    A pro-$p$ group $G$ is said to be represented by a linking tuple $\lambda=(\bX',\bX,\varepsilon,\mu,\omega)$ if 
    \begin{align*}
        G\cong\bigg \langle \tau_{x_1},...,\tau_{x_d}\left| {(\tau_{x_i'})^{\omega_{x_i'}}\lbrack \tau_{x_i'},y_{x_i'}\rbrack=1\text{ for }x_i'\in \bX' \atop \tau_{x_j}^{\omega_{x_j}}\lbrack \tau_{x_j},y_{x_j}\rbrack =1\text{ for } x_j\in \bX }\right.\bigg\rangle 
    \end{align*}
    for some elements $(\tau_{x_i'})_{x_i'\in \bX'}$, $(y_{x_i'})_{x_i'\in \bX'}$ and $(y_{x_j})_{x_j\in \bX}$ in $F\coloneq F(\bX)$ satisfying
    \begin{align*}
        \tau_{x_i'}\equiv \prod_{s=1}^d \tau&_{x_s}^{\mu'(x_i',x_s)}\pmod {F_2},\quad y_{x_j}\equiv \prod_{s=1}^d\tau_{x_s}^{\mu(x_j,x_s)}\pmod{\langle F_2, \tau_{x_j}\rangle }\\
        \text{and}&\qquad y_{x_i'}\equiv \prod_{s=1}^d \tau_{x_s}^{\mu(x_i',x_s)}\pmod{\langle F_2,\Supp_\lambda(x_i')\rangle },
    \end{align*}
    where $\mu':\bX'\times \bX\to \F_p$ is a function satisfying $\mu'(x',x)=0$ iff $\varepsilon(x',x)=0$. 
\end{defi}
Now, let $K$ be a number field not containing~$\zeta_p$, and we fix~$T_K$. Let $S$ be a finite set of tame primes admitting a Koch set~$S'$. We construct a linking tuple $\lambda_S^T$ in the following way. Let $\bX'\coloneq S'$ and $\bX\coloneq S\setminus S'$. Then by Lemma~\ref{lem:Unramified exact sequence} we have 
\begin{align*}
    (K_S^T)^{\rm el,ab}\simeq \bigoplus_{x\in \bX}\F_px.
\end{align*}
For $x'\in \bX'$ fix a generator $\tau_{x'}$ of the inertia group of $x'$ in $(K_S^T)^{\rm el,ab}/K$. For $x\in \bX$ let $\mu(x',x)$ be the $x$-coordinate of $\tau_{x'}$. We also fix a Frobenius $\varphi_{x'}$ in the maximal subextension of $(K_S^T)^{\rm el,ab}/K$ unramified at $x'$ and set $\mu(x',x)$ to be the $x$ coordinate of $\varphi_{x'}$ for $x$ with $\varepsilon(x',x)=0$. It is not hard to see that this value is indeed well defined.

Similarly, for $x\in \bX$ we let $\varphi_x$ be a Frobenius in $(K_{S\setminus \{x\}}^T)^{\rm el,ab}/K$ and let for $y\in S\setminus \{x\}$ be $\mu(x,y)$ the $y$-coordinate of $\varphi_x$. We set $\omega_z\coloneq N(\q_z)-1$, where $\q_{z}$ is the prime ideal corresponding to $z\in \bX\amalg \bX'$.

These constitute $\lambda_S^{T_K}=(\bX',\bX,\varepsilon,\mu,\omega)$. If we want to distinguish the data in the linking tuple from another one, we sometimes write $\mu_S^{T_K}$ instead of $\mu$, etc.. 

The following Theorem shows how $\lambda_S^{T_K}$ and $G_S^{T_K}$ are related:
\begin{theo}\label{theo-Liu}
    We fix~$T_K$. If~$S$ admits a Koch set, then $G_S^{T_K}$ is represented by the linking tuple $\lambda_S^{T_K}$.
\end{theo}
\begin{rema}
    Liu~\cite[Theorem~$1.1$]{Liu24} proved a more general result in a different context, from which Theorem~\ref{theo-Liu} follows. We give a short proof of this case following the lines of Koch~\cite[Section 11]{Koch}. 
\end{rema}
\begin{proof}[Proof of Theorem~\ref{theo-Liu}]
By assumption, we have~$\RB_{S'}^{{T_K}}=1$ and~$G_{K,S'}^{T_K}=1$. Thus, Proposition~\ref{prop:rank of G_S^T} implies that the inertial generators at the primes $\q \in S \setminus S'$ form a minimal system of generators of $G_{K,S}^{T_K}$. By Theorem 6.11 of \cite{Koch}, it remains to prove the existence of an injective map $\Sha_S^{T_K} \to \RB_S^{{T_K}}$, where $\Sha_S^{T_K}$ denotes the kernel of the map $H^2(G_{K,S}^{T_K}) \to \bigoplus_{\p \in S} H^2(\mathcal{G}_\p)$.

The following argument closely follows the proof of~\cite[Theorem~11.3]{Koch}. Thus, we only indicate the necessary modifications. Instead of the subgroup $\mathcal{T}_S$ considered in \cite{Koch}, namely the normal subgroup of the maximal pro-$p$ Galois group $G_K$ of $K$ generated by inertial subgroups at primes outside $S$, we use the larger normal subgroup $\mathcal{T}^{T_K}_S$ generated by the inertia subgroups at primes outside $S$ together with the decomposition subgroups at the primes in $T_K$. In this setting, we obtain an exact sequence $$0 \longrightarrow (\Sha^{T_K}_S)^{\vee} \longrightarrow \mathcal{T}_S^{T_K}/(\mathcal{T}_S^{T_K})^p [\mathcal{T}_S^{T_K}, G_K] \longrightarrow G_K/(G_K)_2.$$ By local class field theory, there exists an epimorphism
{\footnotesize \begin{align*}
     {\prod_{\p \not\in S \cup T_K} E_{\p}/E^p_\p \times \prod_{\p \in T_K} K_\p^{\times}/K_\p^{\times \, p} \cong \prod_{\p \not\in S \cup T_K} \mathcal{T}_\p/\mathcal{T}_\p^p[\mathcal{T}_\p, \mathcal{G}_\p] \times \prod_{\p \in T_K} \mathcal{G}_\p/(\mathcal{G}_\p)_2 \to \mathcal{T}_S^{T_K}/(\mathcal{T}_S^{T_K})^p[\mathcal{T}_S^{T_K}, G_K].}
\end{align*}}
Here $\mathcal{G}_\p$ denotes the Galois group of the maximal pro-$p$ extension of $K_\p$ and $\mathcal{T}_{\p}$ its inertia subgroup.  
We now argue as in the modified version of diagram~\cite[(11.7)]{Koch}, using the isomorphism $$V_S^{S \cup T_K}/K^{\times p} \cong \mathrm{ker}\big ( \prod_{\p \not\in S \cup T_K} E_{\p}/E^p_\p \times \prod_{\p \in T_K} K_\p^{\times}/K_\p^{\times \, p} \longrightarrow J_K/J_K^p K^{\times} \big ),$$ where~$J_K$ denotes the idèle group of~$K$.
\end{proof}

\subsection{Realization of linking data in the number field case}
\label{ssec:Realization of linking data}
In \cite[p. 178]{Labute} Labute asked in case $K=\Q$ --- in our notation --- whether for every linking tuple $\lambda=(\bX,\mu,\omega)$, there exists a set of tame primes $S$ such that  one has $\mu=\mu_S$ for $\lambda_S=(\bX_S,\mu_S,\omega_S)$ under a suitable identification of $S$ and $\bX$.

The following theorem answers his question positively for arbitrary number fields not containing a primitive~$p$-th root of unity. Throughout this subsection, we fix an arbitrary number field $K$ satisfying $\zeta_p \not\in K$ together with $T_K$. 

\begin{theo}\label{thm:Realization of extended linking data}
    Let~$\bX'$ be a finite set of cardinality $d_K'\coloneq r_K+d_p{\rm cl}(K)$ and $\bX$ be an arbitrary finite set of cardinality $d$. If $\lambda=(\bX',\bX,\varepsilon,\mu,\omega)$ is a linking tuple, then there exists a finite tame set of places~$S$ with~$|S|=d$, containing a Koch set, such that, after a suitable identification of $\bX'$ and $S'$ and $\bX$ and $S\setminus S'$, we have:
    \begin{align*}
        \varepsilon_S^{T_K}=\varepsilon,\quad (\mu_S^{T_K})=\mu,\quad v_p((\omega_S^{T_K})_z)\geq v_p(\omega_z),\quad \text{for any}\quad z\in \bX \amalg \bX'.
    \end{align*}
\end{theo}
There is a more general form of the above theorem, for which we need the notation of extension of linking tuples. Let $\lambda=(\bX',\bX,\varepsilon,\mu,\omega)$ and $\widetilde{\lambda}=(\bX',\widetilde{\bX},\widetilde{\epsilon},\widetilde{\mu},\widetilde{\omega})$ be linking tuples. We say that $\widetilde{\lambda}$ extends $\lambda$ if $\bX\subseteq \swidetilde{\bX}$ and the functions $\varepsilon$, $\swidetilde{\varepsilon}$ and $\mu$, $\swidetilde{\mu}$ are equal, wherever both of them are defined. Note that $D_\varepsilon\subseteq D_{\swidetilde{\varepsilon}}$ under these assumptions.
\begin{theo}\label{extension}
    Let $K$ be a number field and fix~$T_K$. We assume that $S$ admits a Koch set~$S'$. Let $\lambda_S^{T_K}$ be the linking tuple associated to it. Consider~$\widetilde{\lambda}=(\bX',\widetilde{\bX},\widetilde{\varepsilon},\widetilde{\mu},\widetilde{\omega})$ a linking tuple extending $\lambda_S^{T_K}$. Then there exists a set~$\swidetilde{S}$ containing $S$ in bijection with $\swidetilde{\bX}\cup \bX'$ such that
    \begin{align*}
        \varepsilon^{T_K}_{\widetilde{S}}=\swidetilde{\varepsilon},\quad \mu_{\widetilde{S}}^{T_K}=\swidetilde{\mu},\quad (\omega_{\widetilde{S}}^{T})_x=\widetilde{\omega}_x\quad \text{for }x\in \bX\amalg \bX'\\
        \text{and}\qquad v_p((\omega_{\widetilde{S}}^T)_x)\geq v_p(\widetilde{\omega}_x)\quad \text{for }x\in \widetilde{\bX}\setminus \bX.
    \end{align*}
\end{theo}
Note that Theorem~\ref{thm:Realization of extended linking data} follows from Lemma~\ref{lem:Existence of Koch Tuples} together with Theorem~\ref{extension}. Furthermore, Theorem~\ref{extension} follows from an inductive application of the following technical lemma, which is based on a method developed by Maire and Sankara in~\cite{MaireSankara}:
\begin{lemm}
    Let $p$ be an odd prime and $K$ a number field with $\zeta_p\not\in K$. Fix~$T_K$. Assume that $S$ admits a Koch set $S'$. Given the following data: a subset $S_r\subseteq S'$; for each $\q\in S\setminus S_r$ an $\ell_\q\in \F_p$, an element~$\tau$ in $(G_S^{T_K})^{\rm el,ab}$, and an integer $m\geq 1$. 

    Then there exists a place $\q'$ of $K$ and a cyclic extension $L/K$ of degree $p$ such that
    \begin{enumerate}
        \item \label{it:Norm condition} $v_p(N(\q')-1)\geq m$;
        \item \label{it:Frobenius condition} $(\q',(K_S^T)^{\rm el,ab}/K)=\tau$;
        \item \label{it:Ramification condition} $L/K$ is ramified exactly at $S_r\cup \{\q'\}$ and totally split in $T_K$;
        \item \label{it:linking condition} there is a generator $\sigma$ of $\Gal(L/K)$ such that $(\q,L/K)=\sigma^{\ell_\q}$ for each $\q\in S\setminus S_r$.
    \end{enumerate}
\end{lemm}
\begin{proof}
    We choose a further place $t_0$ and define $\Theta\coloneq T_K\cup (S\setminus S_r)\cup \{t_0\}$. For a place $v$ we denote by $\GG_v^{\rm ur}\cong \Z_p$ the Galois group of the maximal unramified pro-$p$ extension of $K_v$. We also choose Frobenius lift $\varphi_v$ that generates $\GG_v^{\rm ur}$. Now for $t\in \Theta$ we define  $\chi_t\in H^1(\GG_t^{\rm ur})$ by
    \begin{align*}
        \quad \chi_t(\varphi_t)=\begin{cases}
            0&\text{if }t\in T_K,\\
            \ell_\q&\text{if }t=\q\in S\setminus S_r,\text{ and}\\
            1&\text{if }t=t_0
        \end{cases}
    \end{align*}
    and consider $\chi\coloneq \sum_{t\in \Theta}\chi_t\in \bigoplus_{t\in \Theta}H^1(\GG_t^{\rm ur})$, which is non-zero since $\chi_{t_0}\neq 0$. Let 
    \begin{align*}
        \alpha\coloneq \varphi^\Theta(\chi)\prod_{\q\in S_r}(\q_1,\Gov^{\Theta}/K_1)\in \RB^\Theta\subseteq \Gal(\Gov^{\Theta}/K),
    \end{align*}
    where $\q_1$ denotes a fixed prime of $K_1$ above $\q$ and $\varphi^{\Theta}$ is the map appearing in Lemma \ref{lem:Unramified exact sequence}.
    Note that by Lemma~\ref{lem:Governing field extension generated by Frobenius} the image of $\alpha$ in $\RB_{S_r}^\Theta$ is equal to $\varphi^\Theta_{S_r}(\chi)$. 
    
    Now, let $F$ be the compositum of the fields $\Gov^\Theta$, $K_m=K(\zeta_{p^m})$, and $(K_S^T)^{\rm el,ab}$. Note that by $\zeta_p\not\in K$ we have the following intersections:
    \begin{align*}
        K_1=\Gov^{\Theta}\cap K_m,\quad K=K_m\cap (K_S^{T_K})^{\rm el,ab},\quad \text{and}\quad K=\Gov^\Theta\cap (K_S^{T_K})^{\rm el,ab}.
    \end{align*}
    Clearly $F/K$ is a Galois extension. Due to the observation about the intersection of the fields there exists a $\beta\in \Gal(F/K)$ with the following properties:
    \begin{align*}
        {\rm (i)\quad }\beta\left|_{K_m}\right.=1 \quad  {\rm (ii)\quad }\beta\left|_{(K_S^{T_K})^{\rm el,ab}}\right.=\tau\quad 
        {\rm (iii)\quad }\beta\left|_{\Gov^\Theta}\right.=\alpha
    \end{align*}
    Now, we choose a place $\mathfrak Q$ of $F$, which does not divide any place in $\Theta$ such that $(\mathfrak Q, F/K)=\beta$. We define $\q'$ to be the place of $K$ below $\mathfrak Q$. We now show that conditions \ref{it:Norm condition} to \ref{it:linking condition} are satisfied.

    Condition~\ref{it:Norm condition} follows directly from the fact that $\q$ is completely split in $K_m/K$ by (i) and Condition~\ref{it:Frobenius condition} is a direct consequence of (ii). We next construct the extension $L/K$. For that consider the following commutative diagram with exact rows
    \begin{equation*}
        \begin{tikzcd}[column sep=large]
            H^1(G_{K,S_r})\arrow[d,hook]\arrow[r,"\psi^\Theta_{S_r}"] &\bigoplus_{v\in \Theta}H^1(G_v^{\rm ur})\arrow[r,"\varphi^\Theta_{S_r}"]\arrow[d,equal] &\RB_{S_r}^\Theta\arrow[d,"{\phi_{\frak q'}}",two heads]\\
            H^1(G_{S_r\cup \{\frak q'\}})\arrow[r,"\psi_{S_r\cup \{\frak q'\}}^\Theta"] &\bigoplus_{v\in \Theta}H^1(G_v^{\rm ur})\arrow[r,"\varphi_{S_r\cup \{\frak q'\}}^\Theta"] &\RB_{S_r\cup \{\frak q'\}}^\Theta
        \end{tikzcd}
    \end{equation*}
    Now, by construction of $\q'$ and Lemma~\ref{lem:Governing field extension generated by Frobenius}, we conclude that $\phi_{\q'}(\alpha)=1$. Thus there exists a character $0\neq \widehat{\chi}\in H^1(G_{S_r\cup \{\q'\}})$ such that $\psi^\Theta_{S_r\cup \{\q'\}}(\widehat{\chi})=\chi$. We now let $L$ be the field corresponding to $\ker \widehat{\chi}$. Since $\chi$ and hence $\widehat{\chi}$ is not trivial, we see that $\Gal(L/K)\cong \F_p$. Let $\sigma$ be the generator of $\Gal(L/K)$ mapped to $1$ in $\F_p$ under~$\widehat{\chi}$.

    We immediately see, that for each $\q\in S \setminus S_r$ we have
    \begin{align*}
        \widehat{\chi}\big((\q,L/K)\big)=\chi_\q(\varphi_\q)=\ell_\q.
    \end{align*}
    This shows \ref{it:linking condition}. By a similar argument, we get that each place of $T_K$ is split in $L/K$. Furthermore, since~$L\subseteq K_{S_r\cup \{\q'\}}$, the set of ramified primes of~$L/K$ is contained in~$S_r\cup \{\q'\}$. Moreover,~$L/K$ is ramified at~$\q'$ because~$G_{S'}^T=1$.

    It remains to prove \ref{it:Ramification condition}. By the condition $\chi_t(\varphi_t)=0$ for $t\in T_K$, we have
\begin{equation*}
\alpha|_{\Gov^T}=(\mathfrak Q \cap K_1,\Gov^T/K_1)=\prod_{\q\in S_r} (\q_1,\Gov^T/K_1).
\end{equation*}
Hence, by the Gras-Munnier theorem~\cite{GrasMunnier}, there exists a degree $p$ cyclic extension of $K$ that is precisely ramified at $S_r \cup \{\q'\}$. This extension coincides with $L$, because otherwise we have $d(G_{S' \cup \{\q'\}}^T) \geq d(G^T_{S_r \cup \{\q'\}}) \geq 2$. Since $S'$ is a Koch set, Proposition \ref{prop:rank of G_S^T} implies $d(G_{S' \cup \{\q'\}}^T)=1$, a contradiction.
\end{proof}
\begin{rema}
    If one assumes that $K\cap \Q(\zeta_{p^m})=\Q$, e.g., if $K/\Q$ is unramified at $p$, then it is also possible to prescribe the precise value of $N(\q')-1\pmod{p^m}$ and thus its valuation. The reason for this is that $K(\zeta_{p^m})/K(\zeta_p)\cong 1+p\Z/p^m\subseteq (\Z/p^m)^\times$ and one can choose $\beta|_{K_m}\in \Gal(K_m/K_1)$ in the proof arbitrarily.
    This statement is used by the first author in \cite{Feuerpfeil2026Bockstein}.
\end{rema}

\section{Mildness of $G_{K,S}^T$}
\label{sec:Mildness}
We recall that~$K$ is a number field not containing~$\zeta_p$. We fix~$T_K$ a minimal generating set of ${\rm cl}(K)/p$. We prove Theorem~\ref{first theorem}.

\begin{theo}\label{thm:Mildness for arbitrary fields}
Let~$S_0$ be a finite set of tame primes such that~$|S_0|\geq 2$ and disjoint from~$T_K$. We can find a set $S_1$ of tame primes disjoint from~$T_K$, such that~$|S_1|=2(1+d_K')$ and~$G_{K,S_1 \cup S_0}^{T_K}$ has cohomological dimension $2$ and deficiency~$d_K'$.
\end{theo}

\begin{proof}
We aim to show that~$G_{K,S_0\cup S_1}^{T_K}$ is mild (with respect to some filtration). Let $T_K$ be as in the statement, and choose a Koch set of tame primes $S'$ such that $S'\cap S_0=\emptyset$. Note that~$|S'|= d_K'$, which we abbreviate by $d'$. We write~$S \coloneq S_0\cup S'$. Let~$d \coloneq |S\setminus S'|=|S_0|\geq 2$. By Theorem \ref{theo-Liu}, the group $G_{K,S}^{T_K}$ is represented by a linking tuple $\lambda^{T_K}_S=(\bX', \bX, \varepsilon_S^{T_K}, \mu_S^{T_K}, \omega_S^{T_K})$ associated with the Koch set~$S'$.

We now define a linking tuple $(\bX', \widetilde{\bX}, \widetilde{\epsilon}, \widetilde{\mu}, \widetilde{\omega})$ that extends $\lambda^{T_K}_S$. Recall that~$\bX'\coloneq \{x_1',\dots, x_{d'}'\}$, and~$\bX\coloneq \{x_1,\dots, x_d\}$. We introduce~$\swidetilde{\bX}\coloneq \bX \amalg \{a_1, a_2, b_1,\dots , b_{d'}\}$, and we define the extension~$\swidetilde{\varepsilon}\colon \bX'\times \swidetilde{\bX}\to \{0,1\}$ of~$\varepsilon\coloneq \varepsilon_S^T$ by setting for $x\in \widetilde{\bX}$
    $$\widetilde{\varepsilon}(x_k', x)\coloneq\begin{cases}
        \varepsilon(x_k',x)&\text{if }x\in \bX,\\
        1&\text{if }x=a_2,\text{ and}\\
        0&\text{if }x\neq a_2\in \swidetilde{\bX}\setminus \bX.
    \end{cases}$$
This implies that $\Supp_{\widetilde{\varepsilon}}(x_k')=\Supp_{\varepsilon}(x_k')\cup \{a_2\}$. Thus~$D_{\varepsilon}\subset D_{\widetilde{\varepsilon}}$. For the convenience of the reader, we note that $D_{\widetilde{\varepsilon}}$ is the following (not necessarily disjoint) union:
$$D_{\widetilde{\varepsilon}}= D_{\varepsilon}\cup \Pi(\widetilde{\bX})\cup \{(x_k', b_i), (x_k',a_1)|\quad 1\leq k,i\leq d'\}.$$

Following the general construction, we define a map ~$\widetilde{\mu}\colon D_{\widetilde{\varepsilon}} \to \F_p$ extending~$\mu=\mu_S^{T_K}$ by:  
\begin{align*}
    &\widetilde{\mu}(x_i',b_i)=\widetilde{\mu}(b_i,a_1)\neq 0\qquad &&\text{for }1\leq i\leq d'\\
    &\widetilde{\mu}(x_i,a_1)\neq 0\quad &&\text{for }1\leq i\leq d-1\\
    &\widetilde{\mu}(a_2,x_1)=\widetilde{\mu}(x_d,a_2)=\widetilde{\mu}(a_1,x_d)\neq 0
\end{align*}
and zero for all other pairs in~$D_{\widetilde{\varepsilon}}\setminus D_{\varepsilon}$.

We furthermore set $\widetilde{\omega}_{\tilde{x}}= p^2$ for $\tilde{x}\in \swidetilde{\bX} \setminus \bX$. By Theorem~\ref{extension}, there exists a set of tame primes~$\widetilde{S} \supset S$ with $\widetilde{S} \cap T_K=\emptyset$ such that the linking tuple~$\lambda_{\widetilde{S}}^{T_K}\coloneq (\bX', \widetilde{\bX}, \varepsilon_{\widetilde{S}}^{T_K}, {\mu_{\widetilde{S}}^{T_K}}, \omega_{\widetilde{S}}^{T_K})$ associated with $S'$ satisfies
$$\mu_{\widetilde{S}}^{T_K}=\widetilde{\mu}, \quad \varepsilon_{\widetilde{S}}^{T_K}=\widetilde{\varepsilon}, \quad \nu_p\left( (\omega_{\widetilde{S}}^{T_K})_{\widetilde{x}}\right) \geq \nu_p(\widetilde{\omega}_{\widetilde{x}})=2 \text{ for } \widetilde{x}\in \widetilde{\bX}\setminus \bX .$$
 Note that~$|\swidetilde{S}|=d+2+2d'$. Set~$S_1\coloneq \swidetilde{S}\setminus S_0$; then~$|S_1|=d'+2+d'=2(1+d')$. 

We conclude by showing that~$G_{K,\widetilde{S}}^{T_K}= G_{K,S_0\cup S_1}^{T_K}$ is mild. For this purpose, we apply Theorem~\ref{mildGammaRAAG} to the presentation~$\smash{\PP_{\widetilde{S}}^{T_K}}$ coming from Theorem~\ref{thm:Realization of extended linking data}. We put weight $e_{x_1}=\dots =e_{x_d}=2$ and~$e_{a_i}=e_{b_j}= 1$. Since~$p\geq 3$ and for~$\widetilde{x}\in \swidetilde{\bX}\setminus \bX$ we have~$\nu_p( (\omega_{\widetilde{S}}^T)_{\widetilde{x}})\geq 2$, we infer the following (nontrivial) $d+d'+d'+2=2+d+2d'$ relations:
\begin{align*}
    \begin{array}{ll}
      l_{x_i}\equiv \lbrack \tau_{x_i},\tau_{a_1} \rbrack^{\mu_{\widetilde{S}}^{T_K} (x_i,a_1)} \pmod {F_{e,4}},   & l_{x_d}\equiv \lbrack \tau_{x_d},\tau_{a_2}\rbrack^{\mu_{\widetilde{S}}^{T_K}(x_d,a_2)} \pmod {F_{e,4}} \\
      l_{a_1}\equiv \lbrack \tau_{a_1},\tau_{x_d}\rbrack^{\mu_{\widetilde{S}}^{T_K}(a_1,x_d)} \pmod{F_{e,4}},   & l_{a_2}\equiv \lbrack \tau_{a_2},\tau_{x_1}\rbrack^{\mu_{\widetilde{S}}^{T_K}(a_2,x_1)} \pmod{F_{e,4}},\\
      l_{b_i}\equiv \lbrack \tau_{b_i},\tau_{a_1}\rbrack^{\mu_{\widetilde{S}}^{T_K}(b_i,a_1)} \pmod{F_{e,3}}, & l_{x_k'}\equiv \lbrack \tau_{b_k},\tau_{a_2}\rbrack^{\mu_{\widetilde{S}}^{T_K}(x_k',b_k)\cdot (\mu_{\widetilde{S}}^{T_K})'(x_k',a_2)}\pmod{F_{e,3}}
    \end{array}
\end{align*}

Therefore, the group~$G_{K,\widetilde{S}}^{T_K}$ has an~$(\widetilde{\bX},e)$-RAAG presentation. The underlying graph~$\Gamma$ has $\swidetilde{\bX}\coloneq \{a_1,a_2,b_1,\dots, b_{d'}, x_1,\dots ,x_d\}$, as set of vertices and the~$2+d+d'+d'=d+2+2d'$ edges given by:
$$\bE\coloneq \{ \{ x_1,a_1\},\dots ,\{x_{d-1}, a_1\}, \{x_d, a_{2}\}, \{a_{1},x_d\}, \{a_{2},x_1\}, \{b_{k},a_1\},\{a_2, b_{k}\}\}.$$
Note that the graph~$\Gamma$ is bipartite: for instance, we endow the vertices~$\{a_1,a_2\}$ with blue color and the vertices~$\{b_1,\dots, b_{d'}, x_1,\dots ,x_d\}$ with red color. 

    
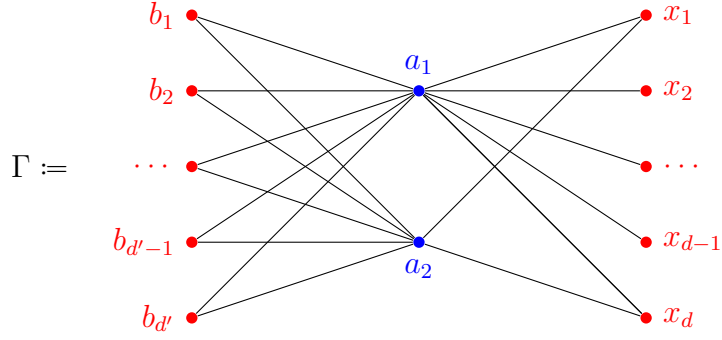
\begin{center}
    \begin{tikzpicture}
    \node at (-2,2) {$\Gamma \coloneq$};

    \node[circle,fill=red,inner sep=1.5pt,label=right:{\color{red}$x_1$}] (x1) at (6,4) {};
    \node[circle,fill=red,inner sep=1.5pt,label=right:{\color{red}$x_2$}] (x2) at (6,3) {};
    \node[circle,fill=red,inner sep=1.5pt,label=right:{\color{red}$\cdots$}] (dots1) at (6,2) {};
    \node[circle,fill=red,inner sep=1.5pt,label=right:{\color{red}$x_{d-1}$}] (xd-1) at (6,1) {};
    \node[circle,fill=red,inner sep=1.5pt,label=right:{\color{red}$x_d$}] (xd) at (6,0) {};

    \node[circle,fill=blue,inner sep=1.5pt,label=above:{\color{blue}$a_1$}] (a1) at (3,3) {};
    \node[circle,fill=blue,inner sep=1.5pt,label=below:{\color{blue}$a_2$}] (a2) at (3,1) {};

    \node[circle,fill=red,inner sep=1.5pt,label=left:{\color{red}$b_1$}] (b1) at (0,4) {};
    \node[circle,fill=red,inner sep=1.5pt,label=left:{\color{red}$b_2$}] (b2) at (0,3) {};
    \node[circle,fill=red,inner sep=1.5pt,label=left:{\color{red}$\cdots$}] (dots2) at (0,2) {};
    \node[circle,fill=red,inner sep=1.5pt,label=left:{\color{red}$b_{d'-1}$}] (bd-1) at (0,1) {};
    \node[circle,fill=red,inner sep=1.5pt,label=left:{\color{red}$b_{d'}$}] (bd) at (0,0) {};

    \draw (a1) -- (xd);
    \draw (a2) -- (xd);
    \draw (a2) -- (x1);

    \draw (b1) -- (a2);
    \draw (b2) -- (a2);
    \draw (bd-1) -- (a2);
    \draw (bd) -- (a2);
    \draw (dots2) -- (a2);

    \draw (b1) -- (a1);
    \draw (b2) -- (a1);
    \draw (bd-1) -- (a1);
    \draw (bd) -- (a1);
    \draw (dots2) -- (a1);

    \draw (x1) -- (a1);
    \draw (x2) -- (a1);
    \draw (xd-1) -- (a1);
    \draw (xd) -- (a1);
    \draw (a1) -- (dots1);
\end{tikzpicture}
    \end{center} 

\justifying
Thus, we conclude from Theorem~\ref{mildGammaRAAG} that~$\PP_{\widetilde{S}}^{T_K}$ is mild. The deficiency of~$G_{\widetilde{S}}^{T_K}$ is~$d+2+2d'-(d+2+d')=d'$.
\end{proof}

\begin{rema}
    Assume that~$S_0$ admits a Koch-type set. Then we can take~$S_1$ of size~$2+d_K'$ such that~$G_{K,S_1\cup S_0}^{T_K}$ is mild. As noted in Proposition~\ref{prop:Probabilty for existence of Koch set} and Remark~\ref{rem:Application existence Koch tuple}, this situation is relatively common for large~$S_0$.
\end{rema}

\begin{rema}\label{rem:sharper result first theorem}
The proof of Theorem~\ref{thm:Mildness for arbitrary fields} shows a sharper result than that~$G_{K, S_0\cup S_1}^{T_K}$ has cohomological dimension~$2$.
There exists a bipartite graph~$\Gamma \coloneq (\widetilde{\bX}, \bE)$, with~$|\widetilde{\bX}|=|S_0|+d_K'+2$ and a weight~$e \in \NN^{|\widetilde{\bX}|}$, such that~$G_{K,S_1\cup S_0}^{T_K}$ is a~$(\widetilde{\bX},e)$-RAAG, and so:
$$\Ll_e(\widetilde{\bX},G_{K,S_0\cup S_1}^{T_K})\simeq \Ll_e(\widetilde{\bX},\Gamma).$$
Furthermore, the graph~$\Gamma$ and the weight~$e$ are both explicitly constructed.
\end{rema}

As a direct consequence, we infer Corollary~\ref{coco Labute}.
\begin{coro}\label{red Lab}
    Assume~$d'_K=0$ (for instance if~$K=\Q$). Let~$S_0$ be a set of tame primes, with cardinality~$d\geq 2$. Then there exists a set~$S_1$ of tame primes of size~$2$ such that~$G_{S_0\cup S_1}$ admits a mild presentation for some~$(\widetilde{\bX},e)$-filtration. In particular~$G_{S_0\cup S_1}$ has cohomological dimension $2$ and deficiency~$0$.
\end{coro}

\begin{rema}
{In the case~$K=\Q$, Theorem~\ref{extension} gives us large freedom in the choice of linking numbers, which are related to the values of some cup products (see for instance~\cite[Proposition~3.9.13]{NSW}). This fact allows us to propose an alternative proof of Corollary~\ref{red Lab} using the criterion of Schmidt~\cite[Theorem~5.5]{Schmidt2007}. We denote by~$\chi_{a_1},\chi_{a_2}, \chi_{x_1},\dots, \chi_{x_d}$ the characters associated to~$\tau_{a_1},\tau_{a_2},\tau_{x_1},\dots, \tau_{x_d}$ in~$H^1(G_{S_0\cup S_1})$. We define by~$V$ the vector space generated by~$\chi_{a_1}, \chi_{a_2}$ and~$U$ the vector space generated by~$\chi_{x_1}, \dots, \chi_{x_d}$. Then the decomposition~$H^1(G_{S_1\cup S_0})=U\oplus V$ satisfies the condition of the criterion. Thus,~$G_{S_0\cup S_1}$ has a mild presentation.}
\end{rema}

\begin{rema}\label{rem:Mildness for arbitrary fields}
The proof of Theorem~\ref{thm:Mildness for arbitrary fields} allows us to show the following results:

    $\bullet$ Assume that~$S_0$ contains only one tame prime. Then, there exists a set~$S_1$ of tame primes of size~$|S_1|=1+2(1+d_K')$ such that~$G_{K,S_1\cup S_0}^{T_K}$ is mild (and has deficiency~$d_K'$).

    $\bullet$ Assume that~$S_0$ is empty. Then, there exists a set~$S_1$ of tame primes of size~$|S_1|=2+2(1+d_K')$ such that~$G_{K,S_0}^{T_K}$ is mild (and has deficiency~$d_K'$). Using Proposition~\ref{app-thm1} we can make better choices of~$S_1$, which allow us to control cup-products. For the special case~$K=\Q$, we refer to Remark~\ref{exa:4-cycle}.
\end{rema}

\section{Genus problems for $G_{K,S}^T$ and graph theory}
\label{sec:Genus Theory}
We follow the terminology of Leoni~\cite{leoni2024zassenhaus}, which we adapt to our context. Let~$\Gamma\coloneq (\bX, \bE)$ be an undirected graph. We define the genus of~$\Gamma$ as the set of pro-$p$ groups:
$${\rm gen}(\Gamma)\coloneq \{G \text{ such that } \Ll(G)\simeq \Ll(\Gamma)\}.$$
We note that the pro-$p$ RAAG~$G_\Gamma$ (defined in Example~\ref{exem: def RAAG}) is always in~${\rm gen}(\Gamma)$. Thus the previous set is never empty. Furthermore, for a fixed~$\Gamma$, all pro-$p$ groups in~${\rm gen}(\Gamma)$ share nice common properties. One of the most notable properties is cohomology. Precisely, as shown for instance in~\cite[Proposition~$1$]{hamza2023extensions}, every pro-$p$ group~$G$ in~${\rm gen}(\Gamma)$ satisfies an isomorphism of graded algebras~$H^\bullet(G)\simeq \A(\Gamma)$.

We now return to the cohomological inverse problem stated in the introduction. We fix a number field~$K$ not containing~$\zeta_p$. For which graphs~$\Gamma$ do there exist a finite set of tame primes~$S$ and a finite disjoint set of primes~$T$ such that~$G_{K,S}^{T}$ is in~${\rm gen}(\Gamma)$? If~$G_{K,S}^{T}$ answers the previous question positively, then we note, using the FAB property, that~$G_{K,S}^{T}$ is not isomorphic to~$G_\Gamma$. 

Proposition~\ref{prop:rank of G_S^T} and Theorem~\ref{theo-Liu} already give some obstructions on the structure of~$\Gamma$. Yet, as noted in the introduction, we currently have no tools to handle the case where the cohomological dimension of~$G_{K,S}^{T}$ is finite but strictly greater than~$2$. This limitation naturally leads us to consider triangle-free graphs, a choice justified by Proposition~\ref{1-RAAG}.

\subsection{Pseudoforest graphs}
This section exhibits a family of graphs relevant to the genus problem. We say that a graph~$\Gamma=(\bX, \bE)$ is a \emph{pseudoforest} if each connected component has at most one cycle. If $|\bX|=|\bE|$, then each connected component has precisely one cycle by Euler's formula. We call those graphs \emph{strict pseudoforests}. We first record the following lemma.

\begin{lemm}\label{lemm-pfst}
The graph $\Gamma=(\bX, \bE)$ is a strict pseudoforest if and only if there exists a map $\varphi_\Gamma\colon \bX \to \bX$ such that
\begin{equation*}
    \bE=\bE_{\varphi_\Gamma} \coloneq \{ \{x_v, \varphi_\Gamma(x_v)\} \mid x_v \in \bX \}
\end{equation*}
and $\varphi_\Gamma^2$ has no fixed points.
\end{lemm}

\begin{proof}
For this proof, we use the following characterization of pseudoforest from~\cite[\S 1.1]{pseudoforest20}. A graph is a pseudoforest if and only if its edges can be oriented so that each vertex has indegree at most one.

First, observe that for any choice of $\varphi_\Gamma$, the map $\varphi_\Gamma^2$ has no fixed point if and only if $|\bX|=|\bE_{\varphi_\Gamma}|$. Let $\Gamma$ be a strict pseudoforest. For each connected component of $\Gamma$, fix an arbitrary orientation on its unique cycle. On each tree attached to the cycle, orient every edge away from the cycle. In this way, we obtain an orientation such that every vertex has indegree $1$. For each vertex $x_v$, define $\varphi_\Gamma(x_v)$ to be the unique vertex having an oriented edge toward $x_v$. 

Conversely, suppose that we are given a function $\varphi_\Gamma$ such that $\bE_{\varphi_\Gamma}=\bE$. We orient each edge from $\varphi_\Gamma(x_v)$ to $x_v$. Then every vertex has indegree $1$. It follows that $\Gamma$ is a pseudoforest.
\end{proof}

A graph $\Gamma=(\bX,\bE)$ is said to be~\emph{inertial} if it contains a strict pseudoforest $(\bX,\bE')$ such that~$\bE'\subseteq \bE$. In particular, for a number field~$K$, the graph~$\Gamma$ is called \emph{$K$-inertial} if it is inertial with $|\bE|\geq d_K' + |\bX|$.

\begin{rema}\label{inertial and map0}
 Assume that~$\Gamma\coloneq (\bX,\bE)$ is $K$-inertial. We denote by~$\Gamma_0\coloneq (\bX,\bE_0)$ a strict pseudoforest satisfying~$\bE_0\subset \bE$. By Lemma~\ref{lemm-pfst}, there is a map $\varphi_\Gamma:\bX\to \bX$, such that $\varphi_\Gamma^2$ has no fixed point, and~$\bE_{\varphi_\Gamma}=\bE_0$. Furthermore, the graph $\Gamma$ is triangle-free if and only if~$\varphi_\Gamma^3$ has no fixed point.
\end{rema}

\subsection{The result for the genus problem in arithmetic}
Let us now state the result of this section.
\begin{prop}\label{app-thm1}
Let $\Gamma\coloneq (\bX, \bE)$ be a triangle-free~$K$-inertial graph. Then there exist disjoint finite sets~$S$ and~$T$ of primes of~$K$ such that~$G_{K,S}^{T_K}$ is in~${\rm gen}(\Gamma)$. Thus~$H^\bullet(G_{K,S}^T)\simeq \A(\Gamma)$.
\end{prop}

\begin{proof}
From Remark \ref{inertial and map0}, the graph $\Gamma$ contains a strict pseudoforest $(\bX, \bE_{\varphi_\Gamma})$ associated with a function $\varphi_\Gamma : \bX \to \bX$ such that~$\bE = \bE_{\varphi_\Gamma} \amalg \bE_1 \amalg \bE_2$, where~$|\bE_1|=d_K'$ and $|\bE_2| \geq 0$. We identify~$\bX'$ with~$\bE_1$. For each~$x'\in \bX'$, we orient the corresponding edge from~$h(x')\in \bX$ to~$t(x')\in \bX$. 
We define a function $\varepsilon\colon \bX' \times \bX \to \{0,1\}$ by 
    \begin{align*}
        \varepsilon(x',x)=\begin{cases}
           1&\text{if }x=h(x'),\\
            0&\text{otherwise.}
        \end{cases}
    \end{align*}
Thus, we have $\Supp_{\varepsilon}(x')=\{h(x')\}$, and so~$D_\varepsilon= \Pi(\bX)\amalg \{(x',x)| x'\in \bX, x\neq h(x')\}$. We define a function $\mu : D_{\varepsilon} \to \F_p$ as follows. For $x, y \in \bX$, we set  
    \begin{align*}
        \mu(x,y)=\begin{cases}
            1 &\text{if }y=\varphi_\Gamma(x)\text{ and}\\
            0&\text{otherwise.}
        \end{cases}
    \end{align*}
Furthermore, for $x' \in \bX'$ and $y \notin \Supp_\varepsilon( x')$, i.e.\ $y\in \bX$ but~$y\neq h(x')$, we set
 \begin{align*}
        \mu(x',y)=\begin{cases}
            1&\text{if }y=t(x')\text{ and}\\
            0&\text{otherwise.}
        \end{cases}
    \end{align*}
By Theorem \ref{extension}, there exists a finite tame set~$S$ admitting a Koch set~$S'$ satisfying $\lambda^{T_K}_S=(\bX', \bX, \varepsilon, \mu, \omega)$ with $\omega=1$. By Theorem \ref{theo-Liu}, the group $G_{K,S}^{T_K}$ admits a minimal presentation
\begin{equation*}
G_{K,S}^{T_K} \cong \left\langle \tau_x \in \bX \;\middle|\; l_x,\ l'_y \ (x\in \bX,\ y\in \bX') \right\rangle
\end{equation*}
satisfying $l_x \equiv \lbrack \tau_x, \tau_{\varphi_\Gamma(x)}\rbrack \pmod{F_2}$ and $l'_y \equiv \lbrack \tau_{h(y)}, \tau_{t(y)}\rbrack \pmod{F_2}$. Here, the latter congruence follows from the assumption~$p\geq 3$.

If $\bE_2 \neq \emptyset$, then we enlarge $T_K$ to $T$ by adding primes $\mathfrak{l}$ whose Frobenius automorphism $\sigma_{\mathfrak{l}} \in G_{K,S}^{T_K}$ satisfies $\sigma_{\mathfrak{l}} \equiv [\tau_{x_1}, \tau_{x_2}] \pmod{F_3}$ for $\{{x_1}, {x_2}\} \in \bE_2$, using the Chebotarev density theorem. One then checks that $G_{K,S}^{T}$ is a $1$-RAAG associated with the graph $\Gamma$. Hence, by Proposition~\ref{1-RAAG}, the group $G_{K,S}^{T}$ is mild, and the desired statements follow.
\end{proof}

\begin{coro}
\label{cor:Genus for d'=0}
    Take~$d_K'=0$. Let~$\Gamma$ be a triangle-free inertial graph. Then there exist disjoint finite sets~$S$ and~$T$ of primes of~$K$ such that~$G_{K,S}^{T}$ is in~${\rm gen}(\Gamma)$. Thus~$H^{\bullet}(G_{K,S}^T) \simeq \A(\Gamma)$.
\end{coro}

    \begin{rema}\label{exa:4-cycle}
        Let $\Gamma$ be a cycle with $4$ vertices, then the associated algebra $\A(\Gamma)$ is not universally Koszul (see~\cite{minac2020enhanced} for a definition of universal Koszul algebras and conjectures). Thus, conjecturally there is no field~$K$, containing a primitive $p$-th root of unity, such that $H^\bullet(G_K)\simeq \A(\Gamma)$ and thus~$G_K$ can not be in~${\rm gen}(\Gamma)$.

        It has already been shown that~$G_\Gamma$ does not occur as~$G_K$ (see~\cite[Theorem~$5.6$]{Quadrelli2014}). The situation with maximal pro-$p$ Galois groups with restricted ramification is very much different. Indeed, by Corollary~\ref{cor:Genus for d'=0} applied to~$K=\Q$, we find a tame~$S$ such that~$G_S$ is in~${\rm gen}(\Gamma)$ and so~$H^\bullet(G_S)\simeq \A(\Gamma)$.
    \end{rema}

\section{Numerical examples and computations}
\label{sec:Examples and Computations}
\subsection{Examples Theorem~\ref{first theorem}}
\label{ssec:Examples mild groups}

It is not hard to find numerical examples of~$2(1+d_K')$ additional tame primes, satisfying Theorem~\ref{first theorem}, in concrete situations. All the computations were made using the computer algebra system OSCAR~\cite{OSCAR}. The relevant source code can be found at \cite{Feuerpfeil2026Code}. Write $S_0=\{q_1,q_2,...,q_d\}$. Then we are looking for tame primes~$q_{a_1}, q_{a_2}$, $q_{1,1},\dots , q_{1,d_K'}$, $q_{2,1},\dots ,q_{2,d_K'}$ such that the the linking numbers satisfy the relations from the proof of Theorem~\ref{thm:Mildness for arbitrary fields}. In our search we were minimizing the size of $q_{a_1}$ and chose the other primes accordingly. We give examples in the rational and quadratic real cases.

\subsubsection{Rational case}
Take~$p=3$. For $S_0=\{7,13,19,37\}$ one finds that the set of tame primes $S_1\coloneq \{10639,826093\}$  completes $S_0$ to satisfy Corollary~\ref{coco Labute}. 
\subsubsection{Real quadratic case}
Fix~$p=3$. We give two examples:

(a) Consider~$K=\Q(\sqrt{3})$, we have~$d_K'=1$. We note that~$S'=\{(5)\}$ yields a Koch set for~$K$. Consider~$S_0\coloneq \{(7),(-\sqrt{3}+4)\}$. Following the proof of Theorem~\ref{thm:Mildness for arbitrary fields} we are able to construct 
\begin{align*}
    S_1\coloneq S'\amalg \{(1063),(593),(531\sqrt{3}+2224)\}
\end{align*}
such that~$G_{K,\swidetilde{S}}$ with~$\widetilde{S}=S_0\cup S_1$ is mild.

(b) For $K=\Q(\sqrt{79})$, we compute that $d'_K=2$ and set $T_K\coloneq \{(3,\sqrt{79}+1)\}$. We set $S_0=\{(7,\sqrt{79}+4),(13,\sqrt{79}+1)\}$. Then $S'=\{(7,\sqrt{79}+3),(11)\}$ is a Koch set disjoint from $S_0$ and we find that for 
\begin{align*}
    S_1\coloneq S'\amalg \big\{&(65\sqrt{79}+66),(17389,\sqrt{79}+4273),\\ &(-2919\sqrt{79}+26594),(86121199, \sqrt{79} + 5709782) \big \}
\end{align*}
and $\swidetilde{S}\coloneq S_0\cup S_1$ the group $G_{K,\widetilde{S}}^{T_K}$ is mild.

\subsection{Examples for $1$-RAAGs}
\label{ssec:Examples for 1 RAAGs}
We give two examples. One in the rational case, and the other in the real quadratic case. We take~$p=3$.

\subsubsection{Rational case}
Consider the case~$K=\Q$. Then~$d_K'=0$ and~$T_K=\emptyset$. Let us consider~$\Gamma$ the~$4$-cycle graph with vertices~$\bX\coloneq \{x_1,x_2,x_3,x_4\}$ and edges~$\bE\coloneq \{\{x_1,x_2\}, \{x_2,x_3\}, \{x_3,x_4\}, \{x_4,x_1\}\}$. Then for $S\coloneq \{7,13,181,5563\}$, the group~$G_S$ is a~$1$-RAAG with respect to~$\Gamma$ and thus its cohomology is~$\A(\Gamma)$.

\subsubsection{Imaginary quadratic case}
Let $K=\Q(\sqrt{-31})$, then~$d_K'=1$. We consider the graph~$\Gamma\coloneq (\bX, \bE)$ on six vertices and seven edges:
$$\bE\coloneq \{ \{x_1,x_2\},  \{x_2,x_3\}, \{x_3,x_4\}, \{x_4,x_5\}, \{x_5,x_6\}, \{x_6,x_1\}, \{x_1,x_4\}\}.$$
The set $T_K\coloneq \{(2,\tfrac{\sqrt{-31}+3}{2})\}$ generates ${\rm cl}(K)\cong \Z/3$ and $S'=\{(7,\sqrt{-31}+2)\}$ is a Koch set. Furthermore, if we set
\begin{align*}
    S\coloneq S'\amalg\{&(7, \sqrt{-31} + 5),
(13),
(4129),
(3613),\\&
(257443, \sqrt{-31} + 245436),
(1062643, \sqrt{-31} + 778479)\}
\end{align*}
then $G_{K,S}^{T_K}$ is a $1$-RAAG for $\Gamma$ and hence $H^\bullet(G_{K,S}^{T_K})\simeq \A(\Gamma)$.

\subsection{Examples which are not~$(\bX,e)$-RAAG}
\label{ssec:nonRAAG}

Fix three finite sets~$\bX=\{x_1,\dots,x_d\}$,~$\bX'\coloneq \{x_1',\dots, x'_{d'}\}$ and~$\bX_0\coloneq \{z_1,\dots z_m\}$. Let $\Ll\coloneq\Ll(\bX).$ We introduce a family~$\rho \coloneq \{\rho_{x}, \rho_{x'}, \rho_z|\quad x\in \bX, x'\in \bX', z\in \bX_0\}$ in~$\Ll_2$. We denote by~$\Ll(\rho)$ the quotient of~$\Ll$ by the Lie-ideal generated by~$\rho$ and~$\A(\rho)$ the quadratic dual of~$\E(\rho)$. Here $\E(\rho)$ 
is the quotient of~$\E\coloneq\E(\bX)$ by the two-sided ideal generated by~$\rho$. For instance, we refer to~\cite[Chapter 1, \S 2]{polishchuk2005quadratic} for definitions and references on quadratic duals. 

Since~$p$ is odd, then~$\Ll_2=\lbrack \Ll_1,\Ll_1\rbrack$. We assume that for every~$x_i \in \bX$, there are coefficients $a_{ij} \in \F_p$ such that~$\rho_{x_i} \coloneq \sum_{x_j\in \bX, x_j\neq x_i} a_{ij}\lbrack X_i,X_j\rbrack$.
For every~$x_k'\in \bX'$, we assume there exists a tuple~$(u_{k},v_{k})\in \Pi(\bX)$ such that~$\rho_{x_k'}\coloneq \lambda_{x_k'}\lbrack \tau_{u_k}, \tau_{v_k}\rbrack$ with~$\lambda_{x_k'}\in \F_p^\ast$. 
This allow us to define a function~$\varepsilon_\rho \colon \bX' \times \bX \to \{0,1\}$ by 
    \begin{align*}
        \varepsilon_\rho(x_k',x)=\begin{cases}
           1&\text{if }x=u_k,\\
            0&\text{otherwise.}
        \end{cases}
    \end{align*}
Thus, we have $\Supp_{\varepsilon_\rho}(x_k')=\{u_k\}$, so~$(x_k',v_k)$ is in~$D_{\varepsilon_\rho}$, and we have:
$$D_{\varepsilon_\rho}\coloneq \Pi(\bX)\cup \{(x_k', x_{k_i})|\quad x_k'\in \bX, x_{k_i}\neq u_k\}.$$
Thus, the family~$\rho$ induces a map~$\mu_\rho\colon D_{\varepsilon_\rho} \to \F_p$ defined by:
$$\mu_{\rho}(x_i,x_j)\coloneq a_{ij} \quad \text{and} \quad \mu_{\rho}(x_k',v_k)\coloneq \lambda_{x_k'}, \quad \text{else } \mu_\rho(x_k',x)\coloneq 0,$$
for $x_i\in \bX$, and~$x_k'\in \bX'$. In particular, we have~$\rho_{x_i} \coloneq \sum_{x_j\neq x_i} \mu_{\rho}(x_i,x_j)\lbrack X_i,X_j\rbrack$.

\begin{prop}\label{theo-arithgenus}
We assume that there exists an order~$\leq$ such that~$\widehat{\rho}\coloneq \{\widehat{\rho_u}\}_{u\in \bX\amalg \bX'\amalg \bX_0}$ is combinatorially free.

Then for any number field~$K$ satisfying~$d'=d_K'$, there exist a set $S$ of tame primes and a disjoint set~$T$ such that
$$\Ll(G^T_{K,S}) \simeq \Ll(\rho), \quad \text{and} \quad H^\bullet(G_{K,S}^T)\simeq \A(\rho).$$
\end{prop}

\begin{proof}
This proof is very similar to the one given in Theorem~\ref{app-thm1}. We fix~$T_K$. By Theorem~\ref{thm:Realization of extended linking data} there exists a tame set~$S$ such that~$\varepsilon_S^{T_K}=\varepsilon_\rho$ and~$\mu_S^{T_K}=\mu_\rho$. 

By the Chebotarev Density Theorem, we can find a set of primes~$T_0\coloneq \{t_z, z\in \bX_0\}$, disjoint from~$T_K$, such that the associated Frobenii~$y_z$ satisfy~$\varPsi_{F(\bX)}(y_{z})-1\equiv \rho_z \pmod{E_3(\bX)}$ for~$z\in \bX_0$. We define~$T\coloneq T_K\cup T_0$. Since the family~$\rho$ is combinatorially free, we conclude that the presentation~$\PP_S^T$ is mild and so:
    $$\Ll(G_{K,S}^T)\simeq \Ll(\rho), \quad \text{and} \quad H^\bullet(G_{K,S}^T)\simeq \A(\rho).$$
\end{proof}

\begin{rema}
    The results on $(\bX,e)$-RAAGs do not require the full strength of Theorem~\ref{thm:Realization of extended linking data}. Only the vanishing of~$\mu_S^{T_K}(x_i,x_j)\neq 0$ is relevant. But Proposition~\ref{theo-arithgenus} heavily depends on the precise values of linking numbers.
\end{rema}


\begin{exem}\label{last exemple}
    Let~$p=3$. We consider various situations for~$K$ a number field not containing~$\zeta_p$, and give numerical examples for Proposition~\ref{theo-arithgenus}.
    
    (a) We fix~$d=4$ and~$d'=0$ so~$\bX\coloneq \{x_1,x_2, x_3,x_4\}$. We define a family~$\rho$ in~$\Ll_2$ by:
    \begin{align*}
        \rho_1\coloneq [X_1,X_4],\quad \rho_2\coloneq &[X_2,X_3+X_4],\quad \rho_3\coloneq[X_3,X_1+X_2+X_4],\\
        \text{and}&\quad \rho_4\coloneq[X_4,X_2+X_3].
    \end{align*}
    The order~$>$ is given by~$X_1> X_2> X_3> X_4$. Thus, $\mu(x_i,x_j)=1$ except for $(x_i,x_j)=(x_1,x_2),(x_1,x_3),(x_2,x_1),(x_4,x_1)$, where it takes the value $0$. Here, we note that~$\A(\rho)$ is presented by four generators~$X_1,X_2,X_3,X_4$ and twelve relations:
$$\{X_i^2, \quad X_uX_v+X_vX_u, \quad X_1X_2, \quad X_2X_3-X_2X_4+X_3X_4, \quad 1\leq i \leq 4, 1\leq u<v \leq 4\}.$$
    Assume that~$K\coloneq \Q$ and~$p=3$. One can explicitly choose $S=\{31,163,409,433\}$, and we have~$H^\bullet(G_S)\simeq \A(\rho)$ and~$\Ll(G_S)\simeq \Ll(\rho)$.

    (b) We set~$\bX\coloneq \{x_1,x_2, x_3,x_4,x_5\}$ and define a family~$\rho$ in~$\Ll_2$ by:
    \begin{align*}
        \rho_1\coloneq [X_1,X_4],\quad \rho_2\coloneq &[X_2,X_3-X_4],\quad \rho_3\coloneq[X_3,X_1+X_2+X_4],
        \\ \rho_4\coloneq[X_4,X_2-X_3], \quad & \rho_5\coloneq \lbrack X_5, X_1\rbrack, \quad \text{and} \quad \rho_{x'} \coloneq \lbrack X_2,X_5\rbrack .
    \end{align*}
The order~$>$ is given by~$X_1>X_2> X_3> X_4> X_5$. 

Here, we note that~$\A(\rho)$ is presented by five generators~$X_1,X_2,X_3,X_4, X_5$ and~$19$ relations:
\begin{multline*}
\{X_i^2, \quad X_uX_v+X_vX_u, \quad X_1X_2, \quad X_3X_5, \quad \\ X_4X_5, \quad X_2X_3+X_2X_4+X_3X_4, \quad 1\leq i \leq 5, 1\leq u<v \leq 5\}.
\end{multline*}
   Take~$p=3$ and $K=\Q(\sqrt{5})$ then $d'_K=1$ and $S'=\{(2)\}$ is a Koch set. If we consider
 \begin{align*}
        S\coloneq \{(2), (283), (4\sqrt{5}-1), (353),\big(\tfrac{-\sqrt{5}-189}{2}\big),  \big(\tfrac{-343\sqrt{5}-17}{2}\big)\},
    \end{align*}
    then $H^\bullet(G_{K,S})\cong \A(\rho)$ and~$\Ll(G_{K,S})\simeq \Ll(\rho)$.
\end{exem}
\bibliography{bibactbib3}
\bibliographystyle{plain}

\end{document}